\documentclass[reqno]{article}
\usepackage{hyperref}
\usepackage{amsthm}
\usepackage{amsfonts}
\usepackage{amsmath}
\usepackage{amssymb}
\usepackage{mathrsfs}
\usepackage[margin=1in]{geometry}
\usepackage{mathtools}
\theoremstyle{definition}
\newtheorem{theorem}{Theorem}
\newtheorem{lemma}[theorem]{Lemma}

\newtheorem{corollary}[theorem]{Corollary}

\newtheorem{definition}[theorem]{Definition}
\def\N{\mathbb{N}}

\def\L{\mathscr{L}}
\def\M{\mathscr{M}}
\def\T{\mathbf{T}}
\def\S{\mathbf{S}}
\def\U{\mathbf{U}}
\def\V{\mathbf{V}}
\def\O{\mathcal{O}}
\def\indices{\mathrm{Indices}}

\def\FV{\mathrm{FV}}
\def\LPA{\L_{\mathrm{PA}}}
\def\LPO{\L^{\mathcal{O}}_{\mathrm{PA}}}

\def\Ord{\mathrm{Ord}}
\def\epom{\epsilon_0\cdot\omega}
\def\indset{\mathcal I}
\def\onset{\mathrm{On}}
\def\Tr{\mathrm{Tr}}

\renewcommand{\Pr}[1]{\T_{#1}\mbox{\small${\vDash}$}}
\newcommand{\Prr}[2]{\T^{#1}_{#2}\mbox{\small${\vDash}$}}
\newcommand{\ucl}[1]{\mathrm{ucl}(#1)}

\newcommand{\case}[1]{\textbf{Case #1:}}

\newcommand{\claim}[1]{\textbf{Claim #1:}}

\newcommand{\str}[1]{\mathrm{Str}(#1)} 

\begin{document}

\title{Self-referential theories}
\author{Samuel A.~Alexander\thanks{Email:
alexander@math.ohio-state.edu}\\
\emph{Department of Mathematics, the Ohio State University}}
\date{August 11, 2020}

\maketitle

\begin{abstract}
We study the structure of families of theories in the language
of arithmetic extended to allow these families to refer to
one another and to themselves.
If a theory contains schemata expressing its own truth and
expressing a specific Turing index for itself,
and contains some other mild axioms, then that theory
is untrue. We exhibit some families of
true self-referential theories that barely avoid this forbidden
pattern.
\end{abstract}

\section{Introduction}

This is a paper about families of r.e.~theories, each capable of referring to
itself and the others.
Many of this paper's results first appeared in the author's dissertation 
\cite{alexanderdissert}.  There, they were stated in terms of families
of interacting mechanical knowing agents.  Here, we will
speak instead of families of self-referential r.e.~theories.  We hope this will
more directly expose the underlying mathematics.

In epistemology, it is well-known that a (suitably idealized) truthful knowing machine capable of
arithmetic, logic, and self-reflection, cannot know its own truth and its own code.
This is due, in various guises,
to authors such as
Lucas \cite{lucas},
Benacerraf \cite{benacerraf}, Reinhardt \cite{reinhardt}, Penrose \cite{penrose},
and Putnam \cite{putnam}.  In terms of self-referential
theories, a true theory satisfying certain assumptions cannot contain schemata stating
its own truth and its own G\"odel number
(if such a theory did exist, we could program a machine knower that knows precisely
its consequences).
Reinhardt conjectured, and Carlson proved \cite{carlson2000}, a truthful machine knower
can know (in a local sense, i.e., expressed by infinite schemata rather
than a single axiom) that it is truthful and has some code, without knowing which.
A true self-referential theory can (in a local sense) state its own truth and
recursive enumerability.  We showed \cite{alexandercode} that, alternatively, a truthful
machine can (in a local sense) exactly know its own code, if not required to know its own
truth.  A true theory can state (in a local sense) its own G\"odel number.

Our goal is to generalize the above consistency results to multiple theories.
The paper contains four main findings.
In the following list of promises, except where otherwise stated, $\prec$ is an r.e.\ well-founded partial-order
on $\omega$, and \emph{expresses} is meant in the local (infinite schema) sense.
\begin{enumerate}
\item There are true theories $(T_i)_{i\in \omega}$
such that $T_i$ expresses a G\"odel number of $T_j$ (all $i,j$)
and $T_i$ expresses the truth of $T_j$ (all $j\prec i$).
\item There are true theories $(T_i)_{i\in \omega}$
such that $T_i$ expresses a G\"odel number of $T_j$ ($j\prec i$),
the truth of $T_j$ ($j\preceq i$),
and the fact that $T_j$ has some G\"odel number (all $i,j$).
\item If $\prec$ is ill-founded, and if we extend
the base language to include a predicate for computable
ordinals and require the theories to include rudimentary facts about
them, then 1 and 2 fail.
\item Finally, if we do not extend the base language as in 3,
then there do exist ill-founded r.e.~partial orders $\prec$ such that 1 and 2 hold.
\end{enumerate}
Our proofs of 1 and 2 are constructive, but the proof of 4 is
nonconstructive.  In short, if 4 were false, either of 1 or 2 could be used
to define the set $WF$ of r.e.~well-founded partial orders of $\omega$
using nothing but arithmetic and a truth predicate $\Tr$ for arithmetic. 
This is impossible since $WF$ is $\Pi^1_1$-complete and $\Tr$ is $\Delta^1_1$.

\section{Preliminaries}
\label{prelimsect}

To us, \emph{theory} and \emph{schema} mean \emph{set of sentences}
(a \emph{sentence}
is a formula with no free variables).

\begin{definition}
\label{stddefn}
(Standard Definitions)
\begin{enumerate}
\item When a first-order structure is clear from context, an \emph{assignment} is a 
function $s$ mapping first-order variables into the universe of that structure.  If 
$x$ is a variable and $u$ is an element of the universe, $s(x|u)$ is the 
assignment that agrees with $s$ except that it maps $x$ to $u$.
\item We write $\M\models\phi[s]$ to indicate that the first-order structure $\M$ 
satisfies the formula $\phi$ relative to the assignment $s$.  We write 
$\M\models \phi$ just in case $\M\models\phi[s]$ for every assignment $s$.
If $T$ is a theory, $\M\models T$ means that $\M\models\phi$ for every 
$\phi\in T$.
\item We write $\FV(\phi)$ for the set of free variables of $\phi$.
\item We write $\phi(x|t)$ for the result of substituting term $t$ for 
variable $x$ in $\phi$.
\item $\LPA$ is the language of Peano arithmetic, with constant 
symbol $0$ and function symbols $S$, $+$, $\cdot$ with the usual arities.
If $\L$ extends $\LPA$, an $\L$-structure \emph{has standard 
first-order part} if it has universe $\N$ and interprets $0$, $S$, $+$ and 
$\cdot$ as intended.
\item We define $\LPA$-terms $\overline n$ ($n\in\N$), called 
\emph{numerals}, so that $\overline 0 \equiv 0$ and $\overline{n+1}\equiv S(\overline 
n)$.
\item We fix a computable bijection 
$\langle 
\bullet,\bullet,\bullet\rangle:\N^3\to\N$.
Being computable, this is $\LPA$-definable, so we may freely 
act as if $\LPA$ contained a function symbol for this bijection. 
Similarly we may act as if $\LPA$ contained a binary predicate symbol $\bullet\in 
W_\bullet$ for membership in the $n$th r.e.~set $W_n$.
\item Whenever a computable language is clear from context, 
$\phi\mapsto\ulcorner\phi\urcorner$ denotes G\"odel numbering.
\item A \emph{valid} formula is one that is true in every structure.
\item A \emph{universal closure} of $\phi$ is a sentence $\forall 
x_1\cdots \forall x_n\phi$ where $\FV(\phi)\subseteq\{x_1,\ldots,x_n\}$.
We write $\ucl{\phi}$ to denote a generic universal closure of $\phi$.
\end{enumerate}
\end{definition}

Note that if $\M$ is a structure and $\psi$ is a universal closure of 
$\phi$, in order to prove $\M\models\psi$ it suffices to let $s$ be an 
arbitrary assignment and show $\M\models\phi[s]$.

To formalize self-referential theories, we employ an extension of
first-order logic where languages may contain new unary connective symbols.
This logic is borrowed from \cite{carlson2000}.

\begin{definition}
\label{baselogicdefn}
(The Base Logic)
A language $\L$ of the \emph{base logic} is a first-order language $\L_0$ together with
a class of symbols called \emph{operators}.
Formulas of $\L$ are defined as usual, with the clause that
$\Pr{i}\phi$ is a formula whenever $\phi$ is a formula and $\Pr{i}$
is an operator.
Syntactic parts of Definition \ref{stddefn} extend to the base logic in obvious
ways (we define $\FV(\Pr{i}\phi)=\FV(\phi)$).
An $\L$-structure $\M$ is a first-order $\L_0$-structure $\M_0$
together with a function that takes one operator $\Pr{i}$, one $\L$-formula
$\phi$, and one assignment $s$, and outputs either True or False---in which case we
write $\M\models\Pr{i}\phi[s]$ or $\M\not\models\Pr{i}\phi[s]$, respectively---satisfying
the following three requirements.
\begin{enumerate}
\item Whether or not $\M\models\Pr{i}\phi[s]$ does not depend on $s(x)$ if $x\not\in \FV(\phi)$.
\item If $\phi$ and $\psi$ are \emph{alphabetic variants}
(meaning
that one is obtained from the other by renaming bound variables so as to respect the binding of
the quantifiers), then $\M\models\Pr{i}\phi[s]$ if and only if $\M\models\Pr{i}\psi[s]$.
\item For variables $x$ and $y$ such that $y$ is substitutable for $x$
in $\Pr{i}\phi$, $\M\models\Pr{i}\phi(x|y)[s]$ if and only if $\M\models \Pr{i}\phi[s(x|s(y))]$.
\end{enumerate}
The definition of $\M\models\phi[s]$ for arbitrary $\L$-formulas is obtained from this by induction.
Semantic parts of Definition \ref{stddefn} extend to the base logic in obvious ways.
\end{definition}

Traditionally the operator $\Pr i$ would be written $K_i$, and the formula $K_i\phi$ would be
read like ``agent $i$ knows $\phi$''.  For the present paper, the added intuition would not be
worth the philosophical distraction.

\begin{theorem}
\label{completecompact}
(Completeness and compactness)  Suppose $\L$ is an r.e.~language in the base logic.
\begin{enumerate}
\item The set of valid $\L$-formulas is r.e.
\item For any r.e. $\L$-theory $\Sigma$, $\{\phi\,:\,\Sigma\models\phi\}$ is r.e.
\item There is an effective procedure, given (a G\"odel number of) an r.e.
$\L$-theory $\Sigma$,
to find (a G\"odel number of) $\{\phi\,:\,\Sigma\models\phi\}$.
\item If $\Sigma$ is an $\L$-theory and $\Sigma\models\phi$, there
are $\sigma_1,\ldots,\sigma_n\in\Sigma$ such that\footnote{We write
$A\rightarrow B\rightarrow C$ for $A\rightarrow (B\rightarrow C)$, and likewise for longer
chains.} $\sigma_1\rightarrow\cdots\rightarrow\sigma_n\rightarrow\phi$ is valid.
\end{enumerate}
\end{theorem}

\begin{proof}
By interpreting the base logic within first-order logic (for details see \cite{alexanderdissert}).
\end{proof}

\begin{definition}
\label{selfrefdefn}
If $\L$ is a first-order language and $I$ is an index set, let $\L(I)$ be the
language (in the base logic) consisting of $\L$ along with 
operators $\Pr{i}$ for all $i\in I$.
\end{definition}

In case $I$ is a singleton, $\LPA(I)$ is a form of Shapiro's \cite{shapiro1985}
language of Epistemic Arithmetic.

\begin{definition}
\label{sentencemaker}
\item
\begin{itemize}
\item
For any $\LPA(I)$-formula $\phi$ with $\FV(\phi)=\{x_1,\ldots,x_n\}$, and for
assignment $s$ (into $\N$), let $\phi^s$ be the sentence
\[
\phi^s \equiv \phi(x_1|\overline{s(x_1)})\cdots(x_n|\overline{s(x_n)})
\]
obtained by replacing all free variables in $\phi$ by numerals for their
$s$-values.
\item
For any language $\mathscr L$ extending $\LPA$, if $\mathscr M$ is an
$\mathscr L$-structure, then $\mathscr M$
is said to \emph{interpret formulas by substitution} if $\mathscr M$ has
standard first-order part and the following
property holds: for every $\mathscr L$-formula $\phi$ and assignment $s$,
$\mathscr M\models\phi[s]$ if and only if $\mathscr M\models\phi^s$.
\end{itemize}
\end{definition}

For example, if $s(x)=0$ and $s(y)=2$
then $(\forall z(x=y+z))^s$ ${\equiv}$ $\forall z(0=S(S(0))+z)$.

\begin{definition}
If $\T=(T_i)_{i\in I}$ is an $I$-indexed family of $\LPA(I)$-theories
and $\mathscr N$ is an $\LPA(I)$-structure, we say $\mathscr N\models\T$
if $\mathscr N\models T_i$ for all $i\in I$.
\end{definition}

\begin{definition}
\label{intendedstructdefn}
Suppose $\T=(T_i)_{i\in I}$ is an $I$-indexed family of 
$\LPA(I)$-theories.
The \emph{intended structure} for $\T$ is the $\LPA(I)$-structure $\M_\T$
with standard first-order part,
interpreting the operators $\Pr{i}$ ($i\in I$) as follows:
\[
\mbox{$\M_\T\models\Pr{i}\phi[s]$ if and only if $T_i\models\phi^s$.}
\]
If $\M_\T\models\T$, we say $\T$ is \emph{true}.
\end{definition}

\begin{lemma}
\label{howintendedworks}
For any family $\T=(T_i)_{i\in I}$ of $\LPA(I)$-theories,
$\M_\T$ interprets formulas by substitution.
\end{lemma}

\begin{proof}
In other words,
we must show that
for every $\LPA(I)$-formula $\phi$ and assignment $s$,
$\M_\T\models\phi[s]$ if and only if $\M_\T\models\phi^s$.
The proof is a straightforward induction.
\end{proof}

\begin{definition}
By the \emph{axioms of Peano arithmetic for $\LPA(I)$} we mean the axioms of Peano arithmetic,
with induction extended to $\LPA(I)$.
\end{definition}

\begin{lemma}
\label{extendedpalemma}
For any $\LPA(I)$-structure $\M$, if $\M$
interprets formulas by substitution,
then $\M$ satisfies the axioms of Peano arithmetic for $\LPA(I)$.
\end{lemma}

\begin{proof}
Let $\M$ be any $\LPA(I)$-structure which interprets formulas by substitution.
This means $\M$ has standard-first order part and for every formula $\phi$
and assignment $s$,
$\M\models \phi[s]$ if and only if $\M\models\phi^s$.

Let $\sigma$ be an axiom of Peano arithmetic for $\LPA(I)$.  If $\sigma$ is not an instance
of induction, then $\M\models\sigma$ since $\M$ has standard first-order part.
But suppose $\sigma$ is $\ucl{\phi(x|0)\rightarrow \forall x(\phi\rightarrow\phi(x|S(x)))\rightarrow\forall x\phi}$.
To see $\M\models\sigma$, let $s$ be an arbitrary assignment and assume
$\M\models\phi(x|0)[s]$ and $\M\models\forall x(\phi\rightarrow \phi(x|S(x)))[s]$.
By assumption, $\M\models\phi^{s(x|0)}$ and $\forall m\in\N$,
if $\M\models \phi^{s(x|m)}$ then $\M\models \phi(x|S(x))^{s(x|m)}$.
Evidently $\phi(x|S(x))^{s(x|m)}\equiv\phi^{s(x|m+1)}$.
By mathematical induction, $\forall m\in\N$, $\M\models\phi^{s(x|m)}$.
By assumption, $\M\models \forall x\phi[s]$.
\end{proof}

\begin{definition}
Suppose $\T=(T_i)_{i\in I}$ is a family $\LPA(I)$-theories.
If $\T^+=(T^+_i)_{i\in I}$ is another such family, we say $\T\subseteq\T^+$
if $T_i\subseteq T^+_i$ for every $i\in I$.
If $T$ is a single $\LPA(I)$-theory,
we say $T\subseteq\T$ if $T\subseteq T_i$ for all $i\in I$.
If $\T^1=(T^1_i)_{i\in I}$ and $\T^2=(T^2_i)_{i\in I}$
are families of $\LPA(I)$-theories, $\T^1\cup\T^2$
is the family $\T'=(T'_i)_{i\in I}$ where each $T'_i=T^1_i\cup T^2_i$.
Arbitrary unions $\bigcup_{n\in X}\T^n$ are defined similarly.
\end{definition}

\begin{definition}
Suppose $\T=(T_i)_{i\in I}$ is a family of $\LPA(I)$-theories.
For each $i\in I$, we say $T_i$ is \emph{$\Pr i$-closed} if
$\Pr i\phi\in T_i$ whenever $\phi\in T_i$.
We say $\T$ is \emph{closed} if each $T_i$ is $\Pr i$-closed.
\end{definition}

\begin{definition}
If $I$ is an r.e.~index set, a family $\T=(T_i)_{i\in I}$ is \emph{r.e.}~just in case
$\{(\phi,i)\,:\,\phi\in T_i\}$ is r.e.
\end{definition}

\section{Generic Axioms}
\label{genericaxiomssection}

If $\T$ is a family of
theories whose truth was in doubt,
and if we state a theorem removing that doubt,
we often state more:
that $\T\cup\S$
is true,
where $\S$ is some
background theory of provability, including non-controversial things
like Peano arithmetic or the schema $\ucl{\Pr i(\phi\rightarrow\psi)\rightarrow\Pr i\phi\rightarrow\Pr i\psi}$.
The choice of $\S$ is somewhat arbitrary, or at best based on tradition.
We will avoid this arbitrary choice by stating results in the form:
``$\T$ is true together with any background theory of provability such that\ldots''

\begin{definition}
\label{closedregenericdefn}
A family $\T$ of $\LPA(\omega)$-theories is
\emph{closed-r.e.-generic}
if $\T$ is r.e.~and
$\M_{\U}\models\T$ for every closed r.e.~family $\U\supseteq\T$ of $\LPA(\omega)$-theories.
\end{definition}

\begin{lemma}
\label{genericclosedreunion}
If $\T$ is a union of closed-r.e.-generic families
and $\T$ is r.e., then $\T$ is closed-r.e.-generic.
\end{lemma}

\begin{proof}
Straightforward.
\end{proof}

\begin{definition}
For $i\in I$ and for $T$ an $\LPA(I)$-theory,
we write $[T]_i$ for the family $\T=(T_k)_{k\in I}$
where $T_i=T$ and $T_k=\emptyset$ for all $k\not=i$.
\end{definition}

\subsection{Closed-r.e.-generic Building Blocks}
\label{closedregenericblocks}

In this subsection, we will exhibit some examples of closed-r.e.-generic families.
They can be combined in diverse ways, via Lemma \ref{genericclosedreunion}, to
form background theories of provability. This will allow us to state
Theorem \ref{onethreethree} below in a generalized way, essentially saying
that a certain doubted theory is consistent with \emph{any} background theory
of provability made up of closed-r.e.-generic building blocks.
The alternative would be for us to arbitrarily choose one such background
theory and build it directly into Theorem \ref{onethreethree}, which would
cause the core details in the proof of Theorem \ref{onethreethree} to get
jumbled up with unimportant distractions.


It is common for a theory to state its own
closure under modus ponens. When there are multiple theories, it
is less clear whether each individual theory should only state its
own closure thereunder, or the closure
of all the other theories, or of some subset thereof.
With the following lemma, we avoid arbitrarily
imposing a decision along these lines.

\begin{lemma}
\label{firstutilbagdeduction}
For any $i,j\in\omega$, the following family is closed-r.e.-generic:
\begin{itemize}
    \item
    $[S]_i$ where $S$ is: ($j$-Deduction) the schema
    $\ucl{\Pr j(\phi\rightarrow\psi)\rightarrow \Pr j\phi\rightarrow\Pr j\psi}$.
\end{itemize}
\end{lemma}

\begin{proof}
Let $\U=(U_k)_{k\in\omega}$ be any closed r.e.~family of $\LPA(\omega)$-theories
such that $\U\supseteq [S]_i$ where $S$ is $j$-Deduction.  We must show
$\M_{\U}\models [S]_i$.  In other words we must show
$\M_{\U}\models \ucl{\Pr j (\phi\rightarrow\psi)
\rightarrow \Pr j\phi \rightarrow \Pr j \psi}$ for any $\phi,\psi$.
Let $s$ be an assignment and assume $\M_{\U}\models \Pr j (\phi\rightarrow\psi)[s]$
and $\M_{\U}\models \Pr j \phi[s]$, we must show $\M_{\U}\models \Pr j \psi[s]$.
By Definition of $\M_{\U}$, $U_j\models (\phi\rightarrow\psi)^s$ and
$U_j\models \phi^s$.  Clearly $(\phi\rightarrow\psi)^s\equiv
\phi^s\rightarrow\psi^s$ so by modus ponens $U_j\models\psi^s$, that is,
$\M_{\U}\models \Pr j \psi[s]$.
\end{proof}

It might not be controversial to require that a theory express its own
ability to prove valid sentences, but in a multi-theory context, should
we require each theory to express that much about all its fellow theories?
The following lemma allows us to avoid arbitrarily declaring the right
answer to that question. Part 2 of this lemma illustrates an interesting combinatorial
property of closed-r.e.-generic building blocks. Some schemas would not be
suitable building blocks by themselves, but when paired with other schemas,
the combination can become a suitable building block.

\begin{lemma}
\label{firstutilbagvalidity}
For any $i,j\in\omega$, the following families are closed-r.e.-generic:
\begin{enumerate}
    \item
    $[S]_i$ where $S$ is: (Assigned Validity) the schema $\phi^s$ ($\phi$ valid,
    $s$ an assignment).
    \item
    $[\mbox{Assigned Validity}]_i\cup [S]_j$
    where $S$ is: ($i$-Validity) $\ucl{\Pr i \phi}$ for $\phi$ valid.
\end{enumerate}
\end{lemma}

\begin{proof}
Both (1) and (2) are r.e.\ by Theorem \ref{completecompact}.

\item
(1)
Let $\U=(U_k)_{k\in\omega}$ be a closed r.e.~superset of $[S]_i$ where $S$
is Assigned Validity.  We must show $\M_{\U}\models [S]_i$.
If $\phi\in [S]_i$ then $\phi$ is $\phi^s_0$ for some valid $\phi_0$ and
some assignment $s$.  Since $\phi_0$ is valid, $\M_{\U}\models \phi_0[s]$.
By Lemma \ref{howintendedworks}, $\M_{\U}\models \phi^s_0$.

\item
(2)
Let $\U=(U_k)_{k\in \omega}$ be any closed r.e.~family
of $\LPA(\omega)$-theories such that $U_i$ contains Assigned Validity
and $U_j$ contains $i$-Validity.
By (1), $\M_\U$ satisfies Assigned Validity.
It remains to show $\M_{\U}$ satisfies $i$-Validity.
Let $\phi$ be valid and $s$ an assignment.
Since $U_i$ contains Assigned Validity, $U_i\models\phi^s$, so
by definition of $\M_{\U}$, $\M_{\U}\models \Pr i\phi[s]$.
\end{proof}

In modal logic, some papers treat the so-called \emph{positive introspection axiom}
(also known as the \emph{KK axiom}) as one of the fundamental axioms of knowledge,
and some do not.
Rather than join either side,
we prefer instead to study the combinatorial
structure of the axiom, asking: are there other schemas we can add to
it to make the combination closed-r.e.-generic?

\begin{lemma}
\label{firstutilbagintrospection}
For any $i,j\in\omega$, the following family is closed-r.e.-generic:
\begin{itemize}
\item $[\mbox{Assigned Validity}]_i\cup[\mbox{$i$-Validity}]_i\cup
[\mbox{$i$-Deduction}]_i\cup [S]_j$ where $S$ is:
\[\mbox{($i$-Introspection) the schema $\ucl{\Pr i\phi\rightarrow\Pr i\Pr i\phi}$.}\]
\end{itemize}
\end{lemma}

\begin{proof}
Recursive enumerability is by Theorem \ref{completecompact}.
Let $\U=(U_k)_{k\in\omega}$ be any closed r.e.~family
of $\LPA(\omega)$-theories such that $U_i$ contains Assigned Validity, $i$-Validity and $i$-Deduction,
and $U_j$ contains $i$-Introspection.
Then $\M_{\U}$ satisfies Assigned Validity and $i$-Validity by
Lemma \ref{firstutilbagvalidity}.
By Lemma \ref{firstutilbagdeduction},
$\M_{\U}$ satisfies $i$-Deduction.
For $i$-Introspection, let $s$ be an assignment and assume
$\M_{\U}\models \Pr i\phi[s]$, we will show $\M_{\U}\models\Pr i\Pr i\phi[s]$.
Since $\M_{\U}\models \Pr i\phi[s]$, $U_i\models\phi^s$.
By Theorem \ref{completecompact}, there are $\sigma_1,\ldots,\sigma_n\in U_i$
such that $\sigma_1\rightarrow\cdots\rightarrow\sigma_n\rightarrow\phi^s$
is valid.
Since $U_i$ contains $i$-Validity,
$U_i\models \Pr i(\sigma_1\rightarrow\cdots\rightarrow\sigma_n\rightarrow\phi^s)$.
By repeated applications of $i$-Deduction contained in $U_i$,
$U_i\models \Pr i\sigma_1\rightarrow\cdots\rightarrow\Pr i\sigma_n\rightarrow \Pr i\phi^s$.
Since $\U$ is closed, $U_i$ is $\Pr i$-closed and
so contains $\Pr i\sigma_1,\ldots,\Pr i\sigma_n$.
So $U_i\models (\Pr i\phi)^s$ and $\M_{\U}\models \Pr i\Pr i\phi[s]$.
\end{proof}

The following lemma shows that arithmetic is generic, which will enable us
to state a later result (Theorem \ref{onethreethree}) in such a way that it is
clear that the result is neither contingent on the presence, nor the absense,
of arithmetic in the theories in question.

\begin{lemma}
\label{firstutilbagarithmetic}
For any $i\in\omega$, $[S]_i$ is closed-r.e.-generic,
where $S$ is the set of axioms of Peano arithmetic for $\LPA(\omega)$.
\end{lemma}

\begin{proof}
By Lemmas \ref{howintendedworks} and \ref{extendedpalemma}.
\end{proof}

Carlson proved \cite{carlson2000} that it is consistent for an idealized
knowing machine to know
``I am a machine'' (without knowing which specific machine it is).
The following lemma sheds additional light: not only is it consistent for a
knowing machine to know ``I am a machine'', in fact that knowledge is
generic: it does not depend heavily on specific arbitrary
decisions about the background theory of provability.

\begin{lemma}
\label{firstutilbagsmt}
For any $i,j\in\omega$, the following family is closed-r.e.-generic.
\begin{itemize}
\item $[S]_i$ where $S$ is: ($j$-SMT)
(See \cite{carlson2000} and \cite{reinhardt}) $\ucl{\exists e\forall x(\Pr 
j\phi\leftrightarrow x\in W_e)}$, $e\not\in\FV(\phi)$.
\end{itemize}
\end{lemma}

\begin{proof}
Suppose $\U=(U_i)_{i\in\omega}$ is a closed r.e.~family of
$\LPA(\omega)$-theories and $\U\supseteq [S]_i$ where $S$ is $j$-SMT.
We must show $\M_{\U}\models [S]_i$.
That is, given $\phi$ with $e\not\in\FV(\phi)$, we must show
$\M_{\U}\models \ucl{\exists e \forall x (\Pr j \phi \leftrightarrow x\in W_e)}$.
Let $s$ be an assignment and let $x_1,\ldots,x_k=\FV(\phi)\backslash\{x\}$.
Since $U_j$ is r.e., by the $S$-$m$-$n$ theorem there is some $n$
such that $W_n=\{m\,:\,U_j\models \phi(x|\overline{m})(x_1|\overline{s(x_1)})
\cdots (x_k|\overline{s(x_k)})\}$.  Since $e\not\in\FV(\phi)$, and $\M_{\U}$ has
standard first-order part, it follows that
$\M_{\U}\models \forall x(\Pr j \phi\leftrightarrow x\in W_e)[s(e|n)]$.
\end{proof}

Finally, the following lemma offers a way to obtain new building blocks from old.
This can be combined with Lemma \ref{firstutilbagsmt}
to advance from ``I am a machine'' to ``I know I am a machine''.

\begin{lemma}
\label{firstutilbagclosure}
For any $i,j\in\omega$ and any closed-r.e.-generic family $\T=(T_k)_{k\in\omega}$,
$\T\cup [S]_i$ is closed-r.e.-generic,
where $S$ is the schema: $\Pr j\phi$ ($\phi\in T_j$).
\end{lemma}

\begin{proof}
Suppose $\U=(U_i)_{i\in\omega}\supseteq \T\cup [S]_i$ is closed and r.e.
Right away $\M_{\U}\models \T$
because $\T$ is closed-r.e.-generic.  It remains to show that
$\M_{\U}\models [S]_i$, i.e., that $\M_{\U}\models S$.
Fix $\phi\in T_j$ and let $s$ be any assignment.  Since $\phi$ is a sentence,
$\phi\equiv\phi^s$ and thus
$T_j\models\phi^s$.  Since $U_j\supseteq T_j$, $U_j\models\phi^s$.
By definition of $\M_{\U}$, $\M_{\U}\models \Pr j\phi[s]$.
\end{proof}

We gather Lemmas \ref{firstutilbagdeduction}--\ref{firstutilbagclosure}
together into the following summary.

\begin{corollary}
\label{closedregenericsummary}
For any $i,j\in\omega$, each of the following families is closed-r.e.-generic.
\begin{enumerate}
    \item
    $[\mbox{$j$-Deduction}]_i$.
    \item
    $[\mbox{Assigned Validity}]_i$.
    \item
    $[\mbox{Assigned Validity}]_i \cup [\mbox{$i$-Validity}]_j$.
    \item
    $[\mbox{Assigned Validity}]_i \cup [\mbox{$i$-Validity}]_i
    \cup [\mbox{$i$-Deduction}]_i \cup [\mbox{$i$-Introspection}]_j$.
    \item
    $[S]_i$ where $S$ is the set of axioms of Peano arithmetic for $\LPA(\omega)$.
    \item
    $[\mbox{$j$-SMT}]_i$.
    \item $\T\cup [S]_i$, for any closed-r.e.-generic $\T$,
    where $S$ is the schema: $\Pr j\phi$ ($\phi\in T_j$).
\end{enumerate}
\end{corollary}

The above building blocks are not exhaustive. In choosing building blocks,
a primary concern was to facilitate creation
of background provability theories strong enough to make our consistency
result (Theorem \ref{onethreethree}) generalize Carlson's
consistency result \cite{carlson2000}. If that were our lone motivation,
we could restrict Corollary \ref{closedregenericsummary} to only those families
where $i=j$, but a secondary motivation was to provide inter-theory versions of
those restricted building blocks.
It would be interesting to
investigate questions about whether the above building-blocks are minimal. For example,
in Lemma \ref{firstutilbagintrospection}, is it really necessary to bundle
$j$-Introspection with all three other schemas? For now, we will leave those questions
open.

\section{First Consistency Result: Prioritizing Exact Codes}

The following theorem fulfills the first promise from the introduction.

\begin{theorem}
\label{onethreethree}
Suppose $\prec$ is an r.e.~well-founded
partial order on $\omega$
and $\T^0=(T^0_i)_{i\in \omega}$ is closed-r.e.-generic.
For each $n\in\N$,
let $\T(n)=(T_i(n))_{i\in \omega}$
where each $T_i(n)$ is the smallest $\Pr i$-closed theory
containing the following:
\begin{enumerate}
\item The axioms in $T^0_i$.
\item $\forall x(\Pr j\phi\leftrightarrow\langle\overline{\ulcorner\phi\urcorner},\overline j,x\rangle\in W_{\overline n})$
whenever $j\in\omega$, 
$\FV(\phi)\subseteq\{x\}$.
\item $\ucl{\Pr j\phi\rightarrow\phi}$ whenever $j\prec i$.
\end{enumerate}
There is some $n\in\N$ such that $\T(n)$ is true.
\end{theorem}

\begin{proof}
By the $S$-$m$-$n$ Theorem, there is a total computable $f:\N\to\N$ such that $\forall n\in\N$,
\[
W_{f(n)}=\{\langle\ulcorner\phi\urcorner,i,m\rangle\,:\,
\mbox{$\FV(\phi)\subseteq\{x\}$ and $T_i(n)\models \phi(x|\overline m)$}\}.\]
Using the Recursion Theorem, fix $n\in\N$ such that
$W_{f(n)}=W_n$.
For brevity write $\T$ for $\T(n)$ and $T_i$ for $T_i(n)$.
We will show $\M_\T\models \T$.  This is a self-referential statement: to show $T_i$
is true includes
showing $\M_\T\models \ucl{\Pr j\phi\rightarrow\phi}$, which is essentially the
statement that $T_j$ is true.  Hence the restriction $j\prec i$, which allows induction
since $\prec$ is well founded.
We will show, by $\prec$-induction on $i$, that $\M_\T\models T_i$
for every $i\in\omega$.
Fix $i\in \omega$ and assume $\M_\T\models T_j$ for all $j\prec i$.
Suppose $\sigma\in T_i$, we will show $\M_\T\models\sigma$.

\item
\case1
$\sigma\in T^0_i$.
Then $\M_\T\models\sigma$ because $\T^0$ is closed-r.e.-generic and
$\T\supseteq\T^0$ is closed r.e.

\item
\case2
$\sigma$ is $\forall x(\Pr j\phi\leftrightarrow
\langle\overline{\ulcorner\phi\urcorner},\overline j,x\rangle\in W_{\overline n})$
for some $j\in\omega$, $\FV(\phi)\subseteq\{x\}$.
Let $s$ be an assignment, $m\in\N$.
The following are equivalent.
\begin{align*}
\M_{\T} &\models \Pr{j}\phi[s(x|m)]\\
T_j &\models \phi^{s(x|m)}
 &\mbox{(Definition of $\M_\T$)}\\
T_j &\models \phi(x|\overline m)
 &\mbox{(Since $\FV(\phi)\subseteq\{x\}$)}\\
\langle \ulcorner\phi\urcorner,j,m\rangle &\in W_n
 &\mbox{(By definition of $n$)}\\
\M_{\T} &\models \langle\overline{\ulcorner\phi\urcorner},\overline j,\overline m\rangle\in W_{\overline n}
 &\mbox{($\M_{\T}$ has standard first-order part)}\\
\M_{\T} &\models \langle\overline{\ulcorner\phi\urcorner},\overline j,x\rangle\in W_{\overline n}[s(x|m)].
 &\mbox{(Lemma \ref{howintendedworks})}
\end{align*}

\item
\case3
$\sigma$ is $\ucl{\Pr j\phi\rightarrow\phi}$ for some $j\prec i$.
Let $s$ be an assignment and assume $\M_\T\models\Pr j\phi[s]$.
This means $T_j\models\phi^s$.
By our $\prec$-induction hypothesis, 
$\M_\T\models T_j$,
so
$\M_\T\models \phi^s$.  By Lemma \ref{howintendedworks},
$\M_\T\models \phi[s]$.

\item
\case4
$\sigma$ is only present in $T_i$ because of the clause that $T_i$
is $\Pr i$-closed.
Then $\sigma$ is $\Pr i\sigma_0$ for some $\sigma_0\in T_i$.
Being in $T_i$, $\sigma_0$ is a sentence, so for any assignment $s$,
$\sigma_0\equiv\sigma^s_0$, $T_i\models\sigma^s_0$,
and finally $\M_\T\models \Pr i\sigma_0[s]$.

\item
By $\prec$-induction, $\M_\T\models T_i$ for all $i\in\omega$.
This shows $\M_\T\models\T$, that is, $\T$ is true.
\end{proof}

The first promise from the introduction is met: for any
r.e.~well-founded partial order $\prec$ on $\omega$, there are theories
$(T_n)_{n\in\omega}$ such that $\forall i,j,k\in\omega$ with $j\prec i$, $T_i$ expresses
the truth of $T_j$, and $T_i$ expresses a G\"odel number of $T_k$.
In order to fulfill the second promise
we will extend Carlson's notion of \emph{stratification} to the case of multiple
operators, and introduce \emph{stratifiers}, a tool used to deal with subtleties
that arise when multiple self-referential theories refer to one another.

In \cite{alexandercode} the technique behind Theorem \ref{onethreethree}
was used to exhibit a machine that knows its own code.

\section{Stratification}
\label{stratificationsection}

For the second promise from the introduction, we need to prove a result
like Theorem \ref{onethreethree} where $T_i$ includes
$\ucl{\Pr j\phi\rightarrow\phi}$ for all $j\preceq i$, not just $j\prec i$.
This rules out the direct $\prec$-induction
of the type used above.  Induction on formula complexity will not work either:
we would need to show all of
$T_i$ consistent just to show $\M_\T\models\Pr i(1=0)\rightarrow(1=0)$.
Instead, we will use ordinal induction.
But there are no ordinals anywhere in sight. To obtain ordinals
to induct on, we will modify the theories we care about, in a process
called \emph{stratification}. We will start with some informal motivational
remarks. Readers who would like to advance directly to the formal definitions
can safely skip Subsection \ref{StratificationMotivationSubsection}.

\subsection{Motivation for Stratification}
\label{StratificationMotivationSubsection}

As explained above, we would like to invoke ordinal induction, but there
are no ordinals in sight.
In order to make ordinal
induction relevant, we will do the following.
We will extend the background language to contain not only
the operators $\Pr i$ ($i\in\omega$), but also operators $\Prr \alpha i$
($i\in\omega$, $\alpha\in\epom$). And instead of focusing directly on $T_i$,
we will focus
on a theory $U_i$ such that the result $U^-_i$ of erasing superscripts from
$U_i$ is $U^-_i=T_i$.
The intended interpretation
of $\Prr \alpha i \phi[s]$ will be
$U_i\cap\alpha\models\phi^s$,
where $U_i\cap\alpha$ is the set of axioms of $U_i$ whose superscripts are
$<\alpha$.
Thus, we may think of $U_i$ as a version of $T_i$ with extra information about
the structure of $T_i$.
We will show (Theorem \ref{stratificationtheorem}), for certain
formulas $\phi$ whose superscripts are positive multiples of $\epsilon_0$,
that $\phi$ holds (in the intended interpretation) if and only if
$\phi^-$ holds.
We will use this, after proving that $U_i$ holds, to conclude that
$T_i$ also holds.

Suppose we would like $T_i$ to contain the axiom $\Pr i(1+1=2)$.
Then, as we carry out the procedure in the above paragraph, we would ensure that
$U_i$ contain all sentences of the form $\Prr \alpha i (1+1=2)$.
This would have the side effect that for any $\beta>\alpha$,
$U_i\cap\beta\models\Prr \alpha i (1+1=2)$,
so that $\Prr {\beta} i \Prr \alpha i (1+1=2)$
would hold in structures with the intended interpretation.

Next, suppose that for every arithmetical sentence $\phi$,
we would like $T_i$ to include
\[\Pr i\phi\rightarrow \Pr i\Pr i\phi.\]
Then we would arrange that
$U_i$ contain
\[\Prr {\alpha} i\phi \rightarrow \Prr {\beta} i \Prr {\alpha} i \phi\]
(whenever $\beta>\alpha$).
The reason for the $\beta$ is as follows.
The intended interpretation of $\Prr {\alpha} i \phi$ shall be
$U_i\cap\alpha\models \phi$. Thus, it would make no sense to put
the axiom
$\Prr {\alpha} i\phi \rightarrow \Prr {\alpha} i \Prr {\alpha} i \phi$
into $U_i$: the fact that $U_i\cap\alpha\models \phi$
does not generally imply that
$U_i\cap\alpha\models\Prr \alpha i\phi$,
since $U_i\cap\alpha$ is limited to formulas in which all superscripts
are $<\alpha$.
At least
$\Prr {\alpha} i\phi \rightarrow \Prr {\beta} i \Prr {\alpha} i \phi$
is plausible.

Again, suppose that for some $j\prec i$, we would like for $T_i$ to include
\[
\Pr i(\Pr j (1=0)\rightarrow (1=0)).
\]
We would arrange
that $U_i$ contain (for all $\alpha$):
\[
\Prr {\alpha} i(\Pr j (1=0)\rightarrow (1=0)).
\]
Note the lack of superscript on $\Pr j$.
The intuition is that $U_i$ is a
version of $T_i$ with extra information about the structure of
$T_i$ (namely, that said structure arises from an
increasing family of theories), but without any additional information
about the structure of $T_j$.

Similarly, suppose we would like $T_i$ to include
\[
\Pr j(\Pr i (1=0))\rightarrow \Pr i(1=0).
\]
We would arrange that $U_i$ contain
(for each $\alpha$):
\[
\Pr j(\Pr i(1=0))\rightarrow \Prr {\alpha} i(1=0).
\]
Note the lack of superscript on the $\Pr i$ within the scope of $\Pr j$.
As above, the intuition is that $U_i$
is a version of $T_i$ with extra information about the structure of
$T_i$. It does not have any extra information about the
structure of $T_j$---not even about what $T_j$ says about $T_i$.
This is important because, when $j\prec i$, we would like $T_i$ to
contain axioms declaring, essentially, the G\"odel number of
$T_j$. This G\"odel number would be hardcoded into such axioms,
and thus there would be no hope of such axioms remaining true if
$T_j$ were changed.

\subsection{Stratification Formal Details}
\label{StratificationDetailsSection}

To get a foothold for induction, instead of considering a particular theory
$T_i$, we will be considering
copies of $T_i$ with ordinal-number superscripts added.
To recover information about the original $T_i$ from these
modified theories, we will need to use sophisticated results
from \cite{carlson1999} about the structure of the ordinals.

\begin{definition}
We define a binary relation $\leq_1$ on $\mathrm{Ord}$ by transfinite recursion
so that for all $\alpha,\beta\in\mathrm{Ord}$, $\alpha\leq_1\beta$ if and only if
$\alpha\leq\beta$ and $(\alpha,\leq,\leq_1)$ is a $\Sigma_1$-elementary substructure
of $(\beta,\leq,\leq_1)$.
\end{definition}

The following theorem is based on calculations from \cite{carlson1999}.
It was used by Carlson to prove Reinhardt's conjecture \cite{carlson2000}.
We state it here without proof.

\begin{theorem}
\label{blackbox}
\item
\begin{enumerate}
\item The binary relation $\leq_1$ is a recursive partial ordering on $\epom$.
\item For all positive integers $m\leq n$, $\epsilon_0\cdot m\leq_1\epsilon_0\cdot n$.
\item For any $\alpha\leq\beta\in\mathrm{Ord}$, $\alpha\leq_1\beta$ if and only if
the following statement is true.  For every finite set $X\subseteq\alpha$
and every finite set $Y\subseteq[\alpha,\beta)$, there is a set
$X<\widetilde Y<\alpha$ such that $X\cup\widetilde Y\cong_{(\leq,\leq_1)}X\cup Y$.
\end{enumerate}
\end{theorem}

The usefulness of Theorem \ref{blackbox} will appear
in Theorem \ref{collapsetheorem}, but first we need some machinery.

\begin{definition}
Let $\indset=((\epom)\times\omega)\sqcup\omega$.
Thus $\LPA(\indset)$ contains operators $\Pr{(\alpha,i)}$ for all $\alpha\in\epom$, $i\in\omega$,
along with operators $\Pr{i}$ for all $i\in\omega$.
As abbreviation,
we write $\Prr{\alpha}{i}$ for $\Pr{(\alpha,i)}$,
and refer to $\alpha$ as its \emph{superscript}.
\end{definition}

\begin{definition}
For any $\LPA(\indset)$-formula $\phi$, $\onset(\phi)\subseteq\epom$ denotes
the set of superscripts appearing in $\phi$.
\end{definition}

\begin{definition}
\label{IStratifiedDefinition}
Suppose $i\in\omega$.
The \emph{$i$-stratified} formulas of $\LPA(\indset)$ are defined as follows (where $\phi$ ranges over $\LPA(\indset)$-formulas).
\begin{enumerate}
\item If $\phi$ is $\Pr j\phi_0$ for some $j\not=i$, then $\phi$ is $i$-stratified if and only if $\phi$ is an $\LPA(\omega)$-formula.
\item If $\phi$ is $\Prr\alpha j\phi_0$ for some $j\not=i$, then $\phi$ is not $i$-stratified.
\item If $\phi$ is $\Pr i\phi_0$, then $\phi$ is not $i$-stratified.
\item If $\phi$ is $\Prr\alpha i\phi_0$, then $\phi$ is $i$-stratified if and only if $\phi_0$ is $i$-stratified and
$\alpha>\onset(\phi_0)$.
\item If $\phi$ is $\neg\phi_0$, $\phi_1\rightarrow\phi_2$, or $\forall x\phi_0$,
then $\phi$ is $i$-stratified if and only if its immediate subformula(s) are.
\item If $\phi$ is atomic, then $\phi$ is $i$-stratified.
\end{enumerate}
An $\LPA(\indset)$-theory $T$ is $i$-stratified if $\phi$ is $i$-stratified whenever $\phi\in T$.
An $\LPA(\indset)$-formula $\phi$ is \emph{very $i$-stratified} if $\phi$ is $i$-stratified
and $\onset(\phi)\subseteq\{\epsilon_0\cdot 1,\epsilon_0\cdot 2,\ldots\}$.
\end{definition}

For example:
\begin{itemize}
\item $\Prr{\omega}7 \Prr 5 7(1=0)\rightarrow \Pr 8(1=0)$ is $7$-stratified but not $6$- or $8$-stratified.
\item $\Prr 5 7 \Prr\omega 7(1=0)$ is not $7$-stratified, nor is $\Prr 5 7 \Pr 7(1=0)$.
\item $\Prr 5 7 \Pr 8 \Pr 7(1=0)$ is $7$-stratified but $\Prr 5 7 \Pr 8 \Prr 4 7(1=0)$ is not.
\end{itemize}

We will not make use of the following lemma, but we state it to further illuminate
Definition \ref{IStratifiedDefinition}.

\begin{lemma}
Suppose $\phi$ is an $\LPA(\indset)$-formula, $i\in\omega$.
Then $\phi$ is $i$-stratified if and only if all of the following conditions hold.
\begin{enumerate}
    \item
    For all $j\in\omega$ and $\alpha\in\epom$, if $\Prr\alpha j$ occurs in $\phi$,
    then $j=i$.
    \item
    Every occurrence of $\Pr i$ in $\phi$ is inside the scope of $\Pr j$ for some
    $j\not=i$.
    \item
    $\Prr\alpha i$ never occurs in $\phi$ inside the scope of $\Pr j$, for any
    $\alpha\in\epom$ or any $j\in\omega$.
    \item
    For all $\alpha,\beta\in\epom$,
    if $\Prr\alpha i$ occurs in $\phi$ inside the scope of $\Prr\beta i$,
    then $\beta>\alpha$.
\end{enumerate}
\end{lemma}

\begin{proof}
Straightforward.
\end{proof}

\begin{definition}
\label{applyinghtophi}
Suppose $X\subseteq\epom$ and $h:X\to \epom$ is order preserving.
For each $\LPA(\indset)$-formula $\phi$,
define an $\LPA(\indset)$-formula $h(\phi)$ inductively as follows:
\begin{enumerate}
    \item If $\phi$ is $\neg\phi_0$, $\phi_1\rightarrow\phi_2$, or
    $\forall x\phi_0$, then $h(\phi)$ is
    $\neg h(\phi_0)$, $h(\phi_1)\rightarrow h(\phi_2)$, or
    $\forall x h(\phi_0)$, respectively.
    \item If $\phi$ is atomic or $\Pr i\phi_0$, then $h(\phi)\equiv \phi$.
    \item If $\phi$ is $\Prr\alpha i\phi_0$ where $\alpha\in X$, then
    $h(\phi)\equiv \Prr{h(\alpha)}i h(\phi_0)$.
    \item If $\phi$ is $\Prr\alpha i\phi_0$ where $\alpha\not\in X$,
    then $h(\phi)\equiv \Prr\alpha i h(\phi_0)$.
\end{enumerate}
\end{definition}

In practice, we will mainly be interested in $\phi$ when $\phi$ is $i$-stratified
for some $i$, in which case $\Prr\alpha j$ cannot occur within the scope of $\Pr k$
in $\phi$ for any $k,j$. For such $\phi$, $h(\phi)$ is simply the result of applying
$h$ to every superscript in $\phi$ that is in $X$.

For example if $X=\{1,\omega\}$, $h(1)=0$, and $h(\omega)=\omega\cdot 2+1$,
then
\[
h\left(\Prr0i(1=0)\rightarrow\Prr1i(1=0)\rightarrow\Prr{\omega} i(1=0)\right)
\,\,\,\,\equiv\,\,\,\,
\Prr0i(1=0)\rightarrow\Prr0i(1=0)\rightarrow\Prr{\omega\cdot 2+1}i(1=0).
\]

In practice, we will primarily be interested in applying
Definition \ref{applyinghtophi} in the case where $\onset(\phi)\subseteq X$.

\begin{definition}
Suppose $X\subseteq\epom$ and $h:X\to\epom$ is order preserving.
For any $\LPA(\indset)$-structure $\mathscr N$,
we define an $\LPA(\indset)$-structure $h(\mathscr N)$ that
has the same universe as $\mathscr N$,
agrees with $\mathscr N$ on $\LPA(\omega)$, and
interprets $\LPA(\indset)\backslash\LPA(\omega)$
so that
\[
\mbox{$h(\mathscr N)\models \Prr\alpha i\phi[s]$
if and only if $\mathscr N\models h(\Prr\alpha i\phi)[s]$.}
\]
\end{definition}

\begin{lemma}
\label{hcommutativitylemma}
Suppose $X\subseteq\epom$, $h:X\to\epom$ is order preserving,
and $\mathscr N$ is an $\LPA(\indset)$-structure.
For any $\LPA(\indset)$-formula $\phi$
and assignment $s$,
$h(\mathscr N)\models\phi[s]$
if and only if $\mathscr N\models h(\phi)[s]$.
\end{lemma}

\begin{proof}
By induction.
\end{proof}

\begin{corollary}
\label{hpreserversvalidity}
Suppose $X\subseteq\epom$ and $h:X\to\epom$ is order preserving.
For any valid $\LPA(\indset)$-formula $\phi$,
$h(\phi)$ is valid.
\end{corollary}

\begin{proof}
For any $\LPA(\indset)$-structure $\mathscr N$ and assignment $s$,
$h(\mathscr N)\models\phi[s]$ by validity, so $\mathscr N\models h(\phi)[s]$
by Lemma \ref{hcommutativitylemma}.
\end{proof}

\begin{definition}
If $X\subseteq\Ord$ and $h:X\to\Ord$, we call $h$ a \emph{covering}
if $h$ is order preserving and whenever $x,y\in X$
and $x\leq_1 y$, $h(x)\leq_1 h(y)$.
\end{definition}

\begin{definition}
Suppose $i\in\omega$.
An $\LPA(\indset)$-theory $T$ is \emph{$i$-unistratified}
if the following conditions hold:
\begin{enumerate}
\item $T$ is $i$-stratified.
\item (Uniformity) Whenever $\phi\in T$,
$X\subseteq\epom$, $\onset(\phi)\subseteq X$, and $h:X\to\epom$ is a covering,
then $h(\phi)\in T$.
\end{enumerate}
\end{definition}

\begin{definition}
If $T$ is an $\LPA(\indset)$-theory and $\alpha\in\epom$,
let $T\cap\alpha$ be the set $\{\phi\in T\,:\,\onset(\phi)\subseteq\alpha\}$ of sentences in $T$ that do not contain
any superscripts $\geq\alpha$.
\end{definition}

\begin{theorem}
\label{collapsetheorem}
(The Collapse Theorem)
Suppose $T$ is an $i$-unistratified $\LPA(\indset)$-theory.
\begin{enumerate}
\item If $n$ is a positive integer and $\onset(\phi)\subseteq\epsilon_0\cdot n$, then $T\models\phi$
if and only if $T\cap(\epsilon_0\cdot n)\models\phi$.
\item If $\alpha\leq_1\beta$ and $\onset(\phi)\subseteq\alpha$, then $T\cap\alpha\models\phi$
if and only if $T\cap\beta\models\phi$.
\end{enumerate}
\end{theorem}

\begin{proof}
Note that since $T$ is $i$-unistratified, in particular $T$ is $i$-stratified.
We will prove (1), the proof of (2) is similar.
\item
($\Leftarrow$) Immediate since $T\cap(\epsilon_0\cdot n)\subseteq T$.

\item
($\Rightarrow$)
Assume $T\models\phi$.
By Theorem \ref{completecompact} there are $\sigma_1,\ldots,\sigma_k\in T$
such that
\[
\Phi \,\equiv\, \sigma_1\rightarrow\cdots\rightarrow \sigma_k\rightarrow\phi
\]
is valid.
Let $X=\onset(\Phi)\cap(\epsilon_0\cdot n)$, $Y=\onset(\Phi)\cap[\epsilon_0\cdot n,\infty)$, note $|X|,|Y|<\infty$.

Since $Y$ is finite, there is some integer $n'>n$ such that $Y\subseteq\epsilon_0\cdot n'$.
By Theorem \ref{blackbox} part 2, $\epsilon_0\cdot n\leq_1\epsilon_0\cdot n'$.
By Theorem \ref{blackbox} part 3, there is some $X<\widetilde Y<\epsilon_0\cdot n$ such that
$X\cup\widetilde Y\cong_{(\leq,\leq_1)}X\cup Y$.

Let $h:X\cup Y\to X\cup\widetilde Y$ be a $(\leq,\leq_1)$-isomorphism.
Since $\onset(\phi)\subseteq\epsilon_0\cdot n$, $h(\phi)=\phi$.
By Corollary \ref{hpreserversvalidity},
\[
h(\Phi)\,\equiv\, h(\sigma_1)\rightarrow\cdots\rightarrow h(\sigma_k)\rightarrow \phi
\]
is valid.
Since $T$ is $i$-unistratified, $h(\sigma_1),\ldots,h(\sigma_k)\in T$.
Finally since $\mathrm{range}(h)<\epsilon_0\cdot n$, $h(\sigma_1),\ldots,h(\sigma_k)\in T\cap(\epsilon_0\cdot n)$,
showing $T\cap(\epsilon_0\cdot n)\models \phi$.
\end{proof}

Loosely speaking, what we have done in Theorem \ref{collapsetheorem} is we have
taken a proof of $\phi$ and we have \emph{collapsed} the proof, shrinking its ordinals
by using Theorem \ref{blackbox} part 3.

\begin{definition}
For every $i\in\omega$ we define the following $\LPA(\indset)$-schema:
\begin{itemize}
\item ($i$-Collapse) $\ucl{\Prr\alpha i\phi\leftrightarrow\Prr\beta i\phi}$ whenever
$\Prr\alpha i\phi$ is $i$-stratified and $\alpha\leq_1\beta$.
\end{itemize}
\end{definition}

\begin{definition}
For any $\LPA(\indset)$-formula $\phi$, $\phi^-$ is the result of erasing 
all superscripts from $\phi$.
If $T$ is an $\LPA(\indset)$-theory, $T^-=\{\sigma^-\,:\,\sigma\in 
T\}$.
\end{definition}

For example, if $\phi$ is $\Prr{\omega}{5}(1=0)\rightarrow\Prr{\omega+1}{5}\Prr{\omega}{5}(1=0)$,
then $\phi^-$ is $\Pr{5}(1=0)\rightarrow\Pr{5}\Pr{5}(1=0)$.

\begin{lemma}
\label{verystratifiableaxioms}
If $T$ is $i$-unistratified then for every $\phi\in T$ there is some $\psi\in T$
such that $\psi$ is very $i$-stratified and $\psi^-\equiv\phi^-$.
\end{lemma}

\begin{proof}
Let $X=\onset(\phi)=\{\alpha_1<\cdots<\alpha_n\}$, $Y=\{\epsilon_0\cdot 1,\ldots,\epsilon_0\cdot n\}$,
and define $h:X\to Y$ by $h(\alpha_j)=\epsilon_0\cdot j$.
Clearly $h$ is order preserving; by Theorem \ref{blackbox} part 2, $h$ is a covering.
Since $T$ is $i$-unistratified, $T$ contains $\psi\equiv h(\phi)$.  Clearly $\psi$ is very $i$-stratified and
$\psi^-\equiv\phi^-$.
\end{proof}

\begin{definition}
For any $\LPA(\omega)$-structure $\mathscr N$,
we define an $\LPA(\indset)$-structure $\mathscr N^-$
that has the same universe as $\mathscr N$, agrees with $\mathscr N$
on $\LPA(\omega)$,
and interprets $\LPA(\indset)\backslash\LPA(\omega)$ as follows.
For any $\LPA(\indset)$-formula $\phi$, $\alpha\in\epom$, $i\in\N$, and assignment $s$,
\[
\mbox{
$\mathscr N^- \models \Prr\alpha i\phi[s]$ if and only if $\mathscr N\models (\Prr\alpha i\phi)^-[s]$.}
\]
\end{definition}

\begin{lemma}
\label{structuregrowingmagic}
Suppose $\mathscr N$ is an $\LPA(\omega)$-structure.
For every $\LPA(\indset)$-formula $\phi$ and assignment $s$,
$\mathscr N^-\models\phi[s]$ if and only if $\mathscr N\models\phi^-[s]$.
\end{lemma}


\begin{proof}
By induction.
\end{proof}

\begin{corollary}
\label{minuspreservesvalidity}
If $\phi$ is a valid $\LPA(\indset)$-formula, then $\phi^-$ is a valid $\LPA(\omega)$-formula.
\end{corollary}

\begin{proof}
Similar to the proof of Corollary \ref{hpreserversvalidity}.
\end{proof}

A converse-like statement holds for Corollary \ref{minuspreservesvalidity} as well.

\begin{lemma}
\label{validitylevator}
For any valid $\LPA(\omega)$-sentence $\phi$ and $i\in\omega$,
there is a valid very $i$-stratified $\LPA(\indset)$-sentence $\psi$
such that $\psi^-\equiv\phi$.
\end{lemma}

\begin{proof}
  Let $\psi\mapsto\psi^+$ be the function taking $\LPA(\omega)$-formulas to
  $\LPA(\indset)$-formulas defined as follows.
  \begin{enumerate}
    \item If $\psi$ is atomic, or of the form $\Pr j\psi_0$ with $j\neq i$,
    then $\psi^+\equiv\psi$.
    \item If $\psi$ is $\Pr i\psi_0$, then $\psi^+ \equiv \Prr {\epsilon_0\cdot n} i \psi_0^+$,
    where $n=\min\{m\in\N\,:\,\epsilon_0\cdot m > \onset(\psi_0^+)\}$.
    \item If $\psi$ is $\neg\psi_0$, $\psi_0\rightarrow\psi_1$, or $\forall x \psi_0$,
    then $\psi^+$ is $\neg\psi_0^+$, $\psi_1^+\rightarrow\psi_2^+$, or $\forall x \psi_0^+$,
    respectively.
  \end{enumerate}
  It is straightforward to show $\phi^+$ is very $i$-stratified.
  We claim $\phi^+$ is valid. Let $\mathscr M$ be any $\LPA(\indset)$-structure,
  we will show $\mathscr M\models\phi^+$.  Let $\mathscr M^+$ be the $\LPA(\omega)$-structure
  with the same universe as $\mathscr M$, which agrees with $\mathscr M$ on the interpretation
  of arithmetic and of $\Pr j$ for $j\not=i$, and which interprets $\Pr i$ as follows:
  \[
    \mbox{
      $\mathscr M^+\models \Pr i \psi[s]$ if and only if $\mathscr M\models (\Pr i\psi)^+[s]$.
    }
  \]
  Since $\phi$ is valid, $\mathscr M^+\models\phi$.
  It follows that $\mathscr M\models\phi^+$.
\end{proof}

\begin{definition}
\label{stratschemasdefn}
Let $i\in\omega$.
We define the following $\LPA(\indset)$-schemas.
\begin{itemize}
\item ($i$-Strativalidity) $\ucl{\Prr\alpha i\phi}$ whenever $\phi$ is a valid $\LPA(\indset)$-formula
and $\Prr\alpha i\phi$ is $i$-stratified.
\item ($i$-Stratideduction) $\ucl{\Prr\alpha i(\phi\rightarrow\psi)\rightarrow\Prr\alpha i\phi\rightarrow \Prr\alpha i\psi}$
whenever this formula is $i$-stratified.
\end{itemize}
\end{definition}

\begin{definition}
\label{straticloseddefn}
An $\LPA(\indset)$-theory $T$ is \emph{$i$-straticlosed} if
the following conditions hold:
\begin{enumerate}
\item $T$ is $i$-unistratified.
\item $T$ includes $i$-Strativalidity, $i$-Stratideduction and $i$-Collapse.
\item For every $\phi\in T$, if $\Prr\alpha i\phi$ is $i$-stratified then
$\Prr\alpha i\phi\in T$.
\end{enumerate}
A family $\T=(T_i)_{i\in\omega}$ is \emph{straticlosed} if each $T_i$ is
$i$-straticlosed.
\end{definition}

The following theorem serves as an omnibus of results from Section 5 of \cite{carlson2000}.

\begin{theorem}
\label{proofstratification}
(Proof Stratification)
Suppose $T$ is an $i$-straticlosed $\LPA(\indset)$-theory.
Then:
\begin{enumerate}
\item Whenever $T\cap\alpha\models\phi$,
$\Prr\alpha i\phi$ is an $i$-stratified sentence,
and $\beta>\alpha$, then $T\cap\beta\models\Prr\alpha i\phi$.
\item For any very $i$-stratified $\LPA(\indset)$-sentences $\rho$ and 
$\sigma$, if $\rho^-\equiv\sigma^-$
then $T\models \rho\leftrightarrow\sigma$.
\item
For any very $i$-stratified $\LPA(\indset)$-sentence $\phi$,
$T\models\phi$ if and only if $T^-\models\phi^-$.
\end{enumerate}
\end{theorem}

\begin{proof}
Note that since $T$ is $i$-straticlosed, in particular $T$
is $i$-unistratified and hence, $i$-stratified.

\item
\claim0
Any time $T\models \Prr\alpha i(\rho\leftrightarrow\sigma)$
and this is $i$-stratified, $T\models \Prr\alpha i\rho
\leftrightarrow \Prr\alpha i\sigma$.

\item
Assume the hypotheses.  By $i$-Strativalidity,
$T\models \Prr\alpha i((\rho\leftrightarrow\sigma)\rightarrow
(\rho\rightarrow\sigma))$.
By $i$-Stratideduction,
\begin{align*}
&T\models \Prr\alpha i((\rho\leftrightarrow\sigma)
\rightarrow(\rho\rightarrow\sigma))\rightarrow
\Prr\alpha i(\rho\leftrightarrow\sigma)\rightarrow
\Prr\alpha i(\rho\rightarrow\sigma)\\
\mbox{and }
&T\models \Prr\alpha i(\rho\rightarrow\sigma)
\rightarrow\Prr\alpha i\rho\rightarrow\Prr\alpha i\sigma.
\end{align*}
It follows that $T\models \Prr\alpha i\rho\rightarrow\Prr\alpha i\sigma$.
The reverse implication is similar.

\item
\claim1
If $T\cap\alpha\models \phi$, $\Prr\alpha i\phi$ is an $i$-stratified sentence,
and $\beta>\alpha$,
then $T\cap\beta\models\Prr\alpha i\phi$.

\item
Given $T\cap\alpha\models\phi$, there are 
$\sigma_1,\ldots,\sigma_n\in T\cap\alpha$ such that
$\sigma_1\rightarrow\cdots\rightarrow\sigma_n\rightarrow\phi$
is valid.
By instances of $i$-Strativalidity and $i$-Stratideduction
contained in $T\cap\beta$,
$T\cap\beta\models \Prr\alpha i\phi$.

\item
\claim2
If $\rho$ and $\sigma$ are very $i$-stratified $\LPA(\indset)$-sentences 
and $\rho^-\equiv\sigma^-$, then $T\models\rho\leftrightarrow\sigma$.

\item
By induction on $\rho$.  Note that $\rho$ is not of the form 
$\Prr\alpha j\rho_0$ (with $j\not=i$), as that is not $i$-stratified.
If $\rho$ is $\Pr j\rho_0$ then $\rho\equiv\rho^-\equiv\sigma^-\equiv\sigma$ and
the claim is immediate.

The only nontrivial remaining case is when $\rho$ is $\Prr\alpha i\rho_0$.
Since $\rho$ is very $i$-stratified, this implies $\alpha=\epsilon_0\cdot 
n$ (some positive integer $n$) and $\rho_0$
is very $i$-stratified.
Since $\sigma^-\equiv\rho^-$
and $\sigma$ is very stratified,
this implies $\sigma\equiv \Prr{\epsilon_0\cdot m}i\sigma_0$
for some positive integer $m$ and very $i$-stratified $\sigma_0$
with $\sigma^-_0\equiv\rho^-_0$.
Assume $m\leq n$, the other case is similar.

By induction, $T\models\rho_0\leftrightarrow\sigma_0$.
By compactness, there is a natural $\ell\geq n$
such that $T\cap 
(\epsilon_0\cdot\ell)\models\rho_0\leftrightarrow\sigma_0$.
By Claim 1, $T\models 
\Prr{\epsilon_0\cdot\ell}i(\rho_0\leftrightarrow\sigma_0)$;
Claim 0 then gives $T\models \Prr{\epsilon_0\cdot\ell}i\rho_0
\leftrightarrow \Prr{\epsilon_0\cdot\ell}i\sigma_0$.
The claim now follows since $T$ contains $i$-Collapse
and $\epsilon_0\cdot m\leq_1\epsilon_0\cdot n\leq_1 \epsilon_0\cdot \ell$ (Theorem \ref{blackbox} part 2).

\item
\claim3
If $\phi$ is an $i$-stratified $\LPA(\indset)$-sentence
and $T\models\phi$, then $T^-\models\phi^-$.

\item
By compactness, find 
$\sigma_1,\ldots,\sigma_n\in T$ such that
$\sigma_1\rightarrow\cdots\rightarrow\sigma_n\rightarrow\phi$
is valid.
By Corollary \ref{minuspreservesvalidity},
so is $\sigma^-_1\rightarrow\cdots\rightarrow\sigma^-_n\rightarrow\phi^-$,
witnessing $T^-\models\phi^-$.

\item
\claim4
If $\phi$ is a very $i$-stratified $\LPA(\indset)$-sentence
and $T^-\models\phi^-$, then $T\models\phi$.

\item
By compactness,
there is a valid sentence
\[
\Phi\,\equiv\,\sigma^-_1\rightarrow\cdots\rightarrow\sigma^-_n\rightarrow\phi^-\]
where each $\sigma_j\in T$.
By Lemma \ref{validitylevator},
there is a valid very $i$-stratified $\LPA(\indset)$-sentence $\Psi$
such that $\Psi^-\equiv\Phi$.
And because $\Psi^-\equiv\Phi$, this implies
\[
\Psi\,\equiv\,\sigma^*_1\rightarrow\cdots\rightarrow\sigma^*_n\rightarrow\phi^*
\]
where each $(\sigma^*_j)^-\equiv\sigma^-_j$, $(\phi^*)^-\equiv\phi^-$,
and $\sigma^*_1,\ldots,\sigma^*_n,\phi^*$ are very $i$-stratified.

By Lemma \ref{verystratifiableaxioms},
there are very $i$-stratified
$\sigma^{**}_1,\ldots,\sigma^{**}_n\in T$ with each $(\sigma^{**}_j)^-\equiv \sigma^-_j\equiv (\sigma^*_j)^-$.
By Claim 2, $T\models\phi^*\leftrightarrow\phi$, and for $j=1,\ldots,n$, $T\models \sigma^{**}_j\leftrightarrow \sigma^*_j$.
Thus
\[
T\models (\sigma^{**}_1\rightarrow\cdots\rightarrow\sigma^{**}_n\rightarrow\phi)\leftrightarrow\Psi,
\]
and since $\Psi$ is valid and the $\sigma^{**}_j\in T$, this shows $T\models\phi$.
\end{proof}

\begin{definition}
\label{stratdefn}
If $\T=(T_i)_{i\in\omega}$ is a straticlosed family of $\LPA(\indset)$-theories,
its
\emph{stratification}, written $\str{\T}$,
is the family $\str{\T}=(S_i)_{i\in\indset}$, where for every $i\in\omega$,
$S_i=T^-_i$ and $\forall\alpha\in\epom$, $S_{(\alpha,i)}=T_i\cap\alpha$.
\end{definition}

\begin{theorem}
\label{stratificationtheorem}
(The Stratification Theorem)
Suppose $\T=(T_i)_{i\in\omega}$ is a straticlosed family of $\LPA(\indset)$-theories.
For any $i\in\omega$, any very $i$-stratified 
$\LPA(\indset)$-formula $\phi$, and any assignment $s$,
$\M_{\str{\T}}\models\phi[s]$
if and only if $\M_{\str{\T}}\models\phi^-[s]$.
\end{theorem}

\begin{proof}
By induction on $\phi$.  The only nontrivial case is
when $\phi$ is $\Prr\alpha i\psi$.
Since $\phi$ is very $i$-stratified, $\psi$ is very $i$-stratified and we may write
$\alpha=\epsilon_0\cdot n$
for some positive integer $n$,
$\onset(\psi)\subseteq\epsilon_0\cdot n$.
The following are equivalent.
\begin{align*}
\M_{\str{\T}} &\models \Prr{\epsilon_0\cdot n}i\psi[s]\\
T_i\cap(\epsilon_0\cdot n) &\models \psi^s
  &\mbox{(Definition of $\M_{\str{\T}}$)}\\
T_i &\models \psi^s
  &\mbox{(Theorem \ref{collapsetheorem})}\\
T^-_i &\models (\psi^s)^-
  &\mbox{(Theorem \ref{proofstratification})}\\
T^-_i &\models (\psi^-)^s
  &\mbox{(Clearly $(\psi^s)^-\equiv(\psi^-)^s$)}\\
\M_{\str{\T}} &\models \Pr i\psi^-[s].
  &\mbox{(Definition of $\M_{\str{\T}}$)}
\end{align*}
\end{proof}

\section{Stratifiers}
\label{stratifiersection}

In order to apply theorems from the previous section, it is necessary to work
with families $\T=(T_i)_{i\in\omega}$ where each $T_i$ is $i$-stratified.
If we want $T^-_i$ to (locally) express the truthfulness of $T^-_j$, we cannot simply
add a schema like $\ucl{\Pr j\phi\rightarrow\phi}$ to $T_i$, because this is not
necessarily $i$-stratified:  for example, the particular instance
$\Pr j\Pr i(1=0) \rightarrow \Pr i(1=0)$
is not $i$-stratified.  But neither is, say, $\Pr j\Prr\alpha i(1=0)\rightarrow \Prr\alpha i(1=0)$,
where $\Prr\alpha i$ occurs within the scope of $\Pr j$.
We will use a schema $\ucl{\Pr j\phi\rightarrow\phi^+}$,
where $\bullet^+$ varies over what we call $i$-stratifiers.

\begin{definition}
\label{stratifierdefn}
Suppose $X\subseteq\epom$, $|X|=\infty$, and $i\in\omega$.
The \emph{$i$-stratifier given by $X$}
is the function $\phi\mapsto\phi^+$ taking $\LPA(\omega)$-formulas to $\LPA(\indset)$-formulas
as follows.
\begin{enumerate}
\item If $\phi$ is atomic or of the form $\Pr j\phi_0$ with $j\not=i$, then $\phi^+\equiv\phi$.
\item If $\phi$ is $\Pr i\phi_0$ then $\phi^+\equiv \Prr\alpha i\phi^+_0$
where
$\alpha=\min\{x\in X\,:\,x>\onset(\phi^+_0)\}$.
\item If $\phi$ is $\neg\psi$, $\psi\rightarrow\rho$, or $\forall x\psi$, then
$\phi^+$ is $\neg\psi^+$, $\psi^+\rightarrow\rho^+$ or $\forall x\psi^+$, respectively.
\end{enumerate}
By an \emph{$i$-stratifier} we mean an $i$-stratifier given by some $X$.
By the \emph{$i$-veristratifier} we mean the $i$-stratifier
given by $X=\{\epsilon_0\cdot1,\epsilon_0\cdot2,\ldots\}$.
\end{definition}

For example, if $\bullet^+$ is the $i$-veristratifier and $j\not=i$ then
\[(\Pr j\Pr i(1=0)\rightarrow \Pr i\Pr i(1=0))^+ \,\,\,\,\equiv\,\,\,\,
\Pr j\Pr i(1=0)\rightarrow\Prr{\epsilon_0\cdot2}i\Prr{\epsilon_0} i(1=0).\]

\begin{lemma}
\label{stratifiermagic}
Suppose $Z\subseteq\epom$,
$h:Z\to \epom$ is order preserving, $i\in\omega$,
and $\bullet^+$ is an $i$-stratifier.
For any $\LPA(\omega)$-formula $\theta$ with $\onset(\theta^+)\subseteq Z$,
there is a computable $i$-stratifier $\bullet^*$ with $\theta^*\equiv h(\theta^+)$.
\end{lemma}

\begin{proof}
Let $X_0=\{h(\alpha)\,:\,\alpha\in\onset(\theta^+)\}$,
let $X=X_0\cup\{\alpha\in\epom\,:\,\alpha>X_0\}$,
and let $\bullet^*$ be the $i$-stratifier given by $X$.
By induction, for every subformula $\theta_0$ of $\theta$, $\theta^*_0\equiv h(\theta^+_0)$.
\end{proof}

\begin{definition}
By a \emph{stratifier-set}, we mean a finite set
\[
    I=\{\bullet^{+_1},\ldots,\bullet^{+_k}\}
\]
where each $\bullet^{+_p}$ is an $i_p$-stratifier for some $i_p\in\omega$,
and $i_1,\ldots,i_k$ are distinct.
With $I$ as above, we write $\indices(I)$ for $\{i_1,\ldots,i_k\}$.
We say $I$ is \emph{computable} if each $\bullet^{+_p}$ is computable.
\end{definition}

For example, if $\bullet^{+_1}$ is a $1$-stratifier,
$\bullet^{+_2}$ is a $5$-stratifier, and $\bullet^{+_3}$ is a $2$-stratifier,
then $I=\{\bullet^{+_1},\bullet^{+_2},\bullet^{+_3}\}$ is a stratifier-set
and $\indices(I)=\{1,5,2\}$. For a non-example, if $\bullet^{*_1}$ and $\bullet^{*_2}$
are distinct $1$-stratifiers, then $\{\bullet^{*_1},\bullet^{*_2}\}$ is
not a stratifier-set, because it fails the distinctness condition.

\begin{definition}
\label{moduloidefn}
\begin{enumerate}
\item
Suppose $\mathscr N$ is an $\LPA(\indset)$-structure
and $I$ is a stratifier-set.
We define an $\LPA(\indset)$-structure $\mathscr N^I$
as follows.
The universe and interpretation of arithmetic of $\mathscr N^I$
agree with those of $\mathscr N$,
as do the interpretations of $\Pr i$ ($i\not\in\indices(I)$)
and $\Prr\alpha i$ (any $\alpha$, $i$).
For each $i\in\indices(I)$, let $\bullet^+\in I$ be the corresponding
$i$-stratifier, and let $\mathscr N^I$ interpret $\Pr i$ as follows.
For any $\LPA(\indset)$-formula $\phi$ and assignment $s$,
we consider two cases.
\begin{enumerate}
  \item If $\phi$ is an $\LPA(\omega)$-formula, then
  $\mathscr N^I \models \Pr i\phi[s]$ if and only if $\mathscr N\models (\Pr i\phi)^+[s]$.
  \item If $\phi$ is not an $\LPA(\omega)$-formula, then
  $\mathscr N^I \models \Pr i\phi[s]$ if and only if $\mathscr N\models \Pr i\phi[s]$.
\end{enumerate}
\item
For any $i\in\omega$, any $i$-stratifier $\bullet^+$, and any
$\LPA(\indset)$-structure $\mathscr N$,
let $\mathscr N^+=\mathscr N^I$ where $I=\{\bullet^+\}$
is the stratifier-set containing only $\bullet^+$.
\end{enumerate}
\end{definition}

Case 1b in Definition \ref{moduloidefn} is somewhat arbitrary.
We will only ever really care about whether
$\mathscr N^I \models \Pr i\phi[s]$ when $\Pr i\phi$ is $j$-stratified for some $j$.
If $\phi$ is not an $\LPA(\omega)$-formula then $\Pr i\phi$ is not $j$-stratified for any $j$.

\begin{lemma}
\label{structurecollapsingmagic}
(Compare Lemma \ref{structuregrowingmagic})
Suppose $\mathscr N$ is an $\LPA(\indset)$-structure, $i\in\omega$, and $\bullet^+$
is an $i$-stratifier.
For every $\LPA(\omega)$-formula $\phi$ and assignment $s$,
$\mathscr N^+\models\phi[s]$ if and only if $\mathscr N\models\phi^+[s]$.
\end{lemma}

\begin{proof}
By induction.
\end{proof}

\begin{lemma}
\label{stratifiersrespectvalidity}
For any $\LPA(\omega)$-formula $\phi$,
any $i\in\omega$, and any $i$-stratifier $\bullet^+$,
$\phi$ is valid if and only if $\phi^+$ is valid.
\end{lemma}

\begin{proof}
\item
($\Rightarrow$)
Assume $\phi$ is valid.
For any $\LPA(\indset)$-structure $\mathscr N$ and assignment $s$,
$\mathscr N^+\models\phi[s]$ by validity, so $\mathscr N\models\phi^+[s]$ by
Lemma \ref{structurecollapsingmagic}.

\item
($\Leftarrow$)
By Corollary \ref{minuspreservesvalidity}.
\end{proof}

\begin{lemma}
\label{decomposingMtotheI}
Suppose $\M$ is an $\LPA(\indset)$-structure,
$I_0$ is a stratifier-set, $i\in\omega$, $i\not\in\indices(I_0)$,
and $\bullet^+$ is an $i$-stratifier.
Let $I=I_0\cup\{\bullet^+\}$.
Then $\M^I=(\M^{I_0})^+$.
Furthermore, $\M^+$ and $\M^I$ agree on the interpretation of $\Pr i$.
\end{lemma}

\begin{proof}
Straightforward.
\end{proof}

\begin{lemma}
\label{technicalLemmaToShortenMainProof}
    Suppose $i\in\omega$ and suppose $\M$ is an $\LPA(\indset)$-structure
    with the property that for every very $i$-stratified $\LPA(\indset)$-formula
    $\phi$ and assignment $s$, $\M\models\phi[s]$ if and only if $\M\models\phi^-[s]$.
    Suppose $I$ is a stratifier-set such that $i\not\in\indices(I)$.
    Then for every very $i$-stratified $\LPA(\indset)$-formula $\phi$ and assignment
    $s$, $\M^I\models\phi[s]$ if and only if $\M^I\models\phi^-[s]$.
\end{lemma}

\begin{proof}
    By induction on $\phi$. Let $s$ be an assignment.
    The only interesting cases are the following.

    \item
    \case1
    $\phi$ is $\Pr j\psi$ for some $j$. Then $\phi^-\equiv\phi$ and the claim
    is trivial.

    \item
    \case2
    $\phi$ has the form $\Prr \alpha j\psi$ for some $j\not=i$. Impossible, this is
    not $i$-stratified.

    \item
    \case3
    $\phi$ has the form $\Prr \alpha i \psi$. The following are equivalent:

    \begin{align*}
        \M^I &\models \Prr \alpha i \psi[s]\\
        \M &\models \Prr \alpha i \psi[s]
            &\mbox{($\M$ and $\M^I$ agree on $\Prr \alpha i$)}\\
        \M &\models (\Prr \alpha i \psi)^-[s]
            &\mbox{(By hypothesis)}\\
        \M^I &\models (\Prr \alpha i\psi)^-[s].
            &\mbox{(Since $i\not\in\indices(I)$, $\M$ and $\M^I$ agree on $\Pr i$)}
    \end{align*}
\end{proof}

\begin{lemma}
\label{raisingMtotheIpreservesintent}
Suppose $\LPA(\indset)$-structure $\M$ is an instance of
Definition \ref{intendedstructdefn}, and suppose
$I$ is a stratifier-set.
Then $\M^I$ interprets formulas by substitution.
\end{lemma}

\begin{proof}
By induction on $|I|$.
If $|I|=0$, we are done by Lemma \ref{howintendedworks}.
Otherwise, we may decompose $I$ as $I=I_0\cup\{\bullet^+\}$
where $\bullet^+$ is an $i$-stratifier.
By induction, $\M^{I_0}$ interprets formulas by substitution ($*$).
By Lemma \ref{decomposingMtotheI}, $\M^I=(\M^{I_0})^+$.

By definition of interpreting formulas by substitution,
for every $\LPA(\indset)$-formula $\phi$ and assignment $s$,
$\M^{I_0}\models\phi[s]$ if and only if $\M^{I_0}\models\phi^s$.
We must show that for every such $\phi$ and $s$, $(\M^{I_0})^+\models\phi[s]$ if and only
if $(\M^{I_0})^+\models\phi^s$.

We induct on $\phi$.
By Definition \ref{moduloidefn}, $(\M^{I_0})$ and $(\M^{I_0})^+$ agree on all symbols
except $\Pr i$, and they agree on $\Pr i \phi_0$ if $\phi_0$ is not an
$\LPA(\omega)$-formula.
Thus the only nontrivial case is when $\phi$
is of the form $\Pr i\phi_0$ for some $\LPA(\omega)$-formula $\phi_0$.
Any such $\phi$ is itself an $\LPA(\omega)$-formula and thus susceptible to
Lemma \ref{structurecollapsingmagic}.  The following are equivalent.
\begin{align*}
(\M^{I_0})^+ &\models \phi[s]\\
\M^{I_0} &\models \phi^+[s]
  &\mbox{(Lemma \ref{structurecollapsingmagic})}\\
\M^{I_0} &\models (\phi^+)^s
  &\mbox{(By ($*$))}\\
\M^{I_0} &\models (\phi^s)^+
  &\mbox{(Clearly $(\phi^+)^s\equiv(\phi^s)^+$)}\\
(\M^{I_0})^+ &\models \phi^s.
  &\mbox{(Lemma \ref{structurecollapsingmagic})}
\end{align*}
\end{proof}

\section{Generic Stratified Axioms}

We now have enough technical machinery to fulfill the second promise from the
Introduction.
We will fulfill it in a general
way, essentially saying: ``The theories in question, whose
truth were in doubt, are true together with any background theory of
provability such that...'' Just like in Section \ref{genericaxiomssection},
we do this by introducing a notion of genericness.
Throughout this section, $\prec$ is an r.e.\ well-founded partial-order
of $\omega$.


\begin{definition}
If $i\in\omega$, we say that a stratifier-set $I$ is \emph{above $i$}
if $\forall j\in\indices(I)$, $i\prec j$. We adopt the following convention:
if $I$ is above $i$ then we will write $I$ as $I(i)$ in order to remind
ourselves that $I$ is above $i$.
\end{definition}

\begin{definition}
\label{bootstrapclosedrestrat}
(Compare Definition \ref{closedregenericdefn})
Suppose $\T=(T_i)_{i\in\omega}$
is an r.e.~family of $\LPA(\indset)$-theories
and each $T_i$ is $i$-unistratified.
We say $\T$ is \emph{$\prec$-straticlosed-r.e.-generic}
(or \emph{straticlosed-r.e.-generic}, if $\prec$ is clear from context) if
for every straticlosed r.e.~family $\U\supseteq\T$, every $i\in\omega$,
and every computable stratifier-set $I(i)$ above $i$,
$\M^{I(i)}_{\str\U}\models T_i$.
\end{definition}

\begin{lemma}
\label{straticlosedstitching}
If the family $\T=(T_i)_{i\in\omega}$ of $\LPA(\indset)$-sets is r.e.~and is a union
of straticlosed-r.e.-generic families, then $\T$ is straticlosed-r.e.-generic.
\end{lemma}

\begin{proof}
Straightforward.
\end{proof}

\subsection{Straticlosed-r.e.-generic Building Blocks}

As in Section \ref{closedregenericblocks}, we exhibit some examples of
straticlosed-r.e.-generic families, which can be combined (via
Lemma \ref{straticlosedstitching}) to form background theories of
provability. This will allow us to state Theorem \ref{generalizedtwoonethree}
below in a generalized way, essentially saying that certain doubted
theories are consistent with \emph{any} background theory of provability built up
from such blocks. This saves us from having to arbitrarily impose any
particular background theory of provability.

In the following lemma, for part 3, the intuition is that for the purpose
of straticlosed-r.e.-genericness, what things
$T_i$ says about $T_j$ need not merely be true,
but must even remain true when a $j$-stratifier is applied to them.
$\Pr j(\phi\rightarrow\psi)
\rightarrow\Pr j\phi \rightarrow\Pr j\psi$ lacks this property, because
it could be that $(\Pr j\phi)^+\equiv \Prr\alpha j\phi^+$,
$(\Pr j\psi)^+\equiv \Prr\beta j\psi^+$, where $\beta<\alpha$.
For parts 1--2, the reason we cannot merge these parts
into $[\mbox{$j$-Deduction}]_i$ ($j\preceq i$) is because
$[\mbox{$i$-Deduction}]_i$ is not $i$-stratified.

\begin{lemma}
\label{secondutilbagdeduction}
(Compare Lemma \ref{firstutilbagdeduction})
For any $i,j\in\omega$, each of the following families is straticlosed-r.e.-generic.
\begin{enumerate}
\item $[\mbox{$i$-Stratideduction}]_i$.
\item $[\mbox{$j$-Deduction}]_i$ (if $j\prec i$).
\item $[S]_i$ (if $i\prec j$) where $S$ is the following schema ($\phi,\psi$
    range over $\LPA(\omega)$-formulas):
\[
  \mbox{
    (Modified $j$-Deduction)
    $\ucl{\Pr j(\phi\rightarrow\psi)\rightarrow\Pr j\phi\rightarrow
    \Pr j(\psi\wedge\phi)}$.
  }
\]
\end{enumerate}
\end{lemma}

\begin{proof}
Clearly these families are unistratified.
Recursive enumerability follows from the fact that $\prec$ is r.e.
In each case below, let $\U=(U_k)_{k\in\omega}$ be a straticlosed r.e.~family
extending the family in question.
For brevity, let $\M=\M_{\str\U}$.

\item
(1) Let $I(i)$ be any computable stratifier-set above $i$, we must show
$\M^{I(i)}\models \ucl{\Prr\alpha i(\phi\rightarrow\psi)
\rightarrow\Prr\alpha i\phi\rightarrow\Prr\alpha i\psi}$ assuming this
formula is $i$-stratified.
Let $s$ be an assignment and assume
$\M^{I(i)}\models \Prr\alpha i(\phi\rightarrow\psi)[s]$
and $\M^{I(i)}\models\Prr\alpha i\phi[s]$.
By Definition \ref{moduloidefn}, $\M^{I(i)}$ and $\M$ agree on $\Prr\alpha i$,
so $\M\models\Prr\alpha i(\phi\rightarrow\psi)[s]$.
By definition of $\M=\M_{\str\U}$, this means $U_i\cap\alpha\models (\phi\rightarrow\psi)^s$.
Clearly $(\phi\rightarrow\psi)^s\equiv \phi^s\rightarrow\psi^s$,
so $U_i\cap\alpha\models \phi^s\rightarrow\psi^s$.
By similar reasoning, $U_i\cap\alpha\models\phi^s$.
By modus ponens, $U_i\cap\alpha\models\psi^s$, which means
$\M\models \Prr \alpha i \psi[s]$.
Since $\M$ and $\M^{I(i)}$ agree on $\Prr\alpha i$,
$\M^{I(i)}\models \Prr\alpha i\psi[s]$, as desired.

\item
(2) Let $I(i)$ be any computable stratifier-set above $i$, we must show
$\M^{I(i)}\models \ucl{\Pr j(\phi\rightarrow\psi)
\rightarrow\Pr j\phi\rightarrow\Pr j\psi}$.

Let $s$ be an assignment and assume
$\M^{I(i)}\models \Pr j(\phi\rightarrow\psi)[s]$
and $\M^{I(i)}\models\Pr j\phi[s]$.
Since $I(i)$ is above $i$ and $j\prec i$, $\M^{I(i)}$ and $\M$
agree on $\Pr j$, so
$\M\models \Pr j(\phi\rightarrow\psi)[s]$
and $\M\models\Pr j\phi[s]$.
By definition of $\M$, $U^-_j\models\phi^s\rightarrow\psi^s$
and $U^-_j\models\phi^s$, thus $U^-_j\models\psi^s$, so
$\M\models \Pr j\psi[s]$ and thus so does $\M^{I(i)}$.

\item
(3) Let $I(i)$ be any computable stratifier-set above $i$,
we must show
$\M^{I(i)}\models \ucl{\Pr j(\phi\rightarrow\psi)\rightarrow
    \Pr j\phi\rightarrow \Pr j(\psi\wedge\phi)}$.
Let $s$ be an assignment and assume $\M^{I(i)}\models \Pr j(\phi\rightarrow\psi)[s]$
and $\M^{I(i)}\models\Pr j\phi[s]$.
If $j\not\in\indices(I(i))$, then $\M^{I(i)}$ and $\M$
agree on $\Pr j$, so reason as in (2) above.
If not, we can write $I(i)=I_0\cup\{\bullet^+\}$ where $\bullet^+$ is a
computable $j$-stratifier, and Lemma \ref{decomposingMtotheI} ensures that
$\M^{I(i)}$ and $\M^+$ agree on $\Pr j$.
By definition of $\M^+$, $\M\models(\Pr j(\phi\rightarrow\psi))^+[s]$
and $\M\models(\Pr j\phi)^+[s]$.
Let $\alpha,\beta\in\epom$ be such that
$(\Pr j(\phi\rightarrow\psi))^+\equiv\Prr\alpha j(\phi^+\rightarrow\psi^+)$
and $(\Pr j\phi)^+\equiv\Prr\beta j\phi^+$.
Then
$\M\models\Prr\alpha j (\phi^+\rightarrow\psi^+)[s]$ and $\M\models\Prr\beta j\phi^+[s]$.
This means
$U_j\cap\alpha\models(\phi^+\rightarrow\psi^+)^s$ and $U_j\cap\beta\models(\phi^+)^s$.
Since $\phi$ is a subformula of $\phi\rightarrow\psi$, it follows $\beta\leq\alpha$, thus
$U_j\cap\alpha\models (\psi^+ \wedge \phi^+)^s$.
So $\M\models\Prr\alpha j(\psi^+\wedge\phi^+)[s]$.
By Definition \ref{stratifierdefn},
\[
\Prr\alpha j(\psi^+\wedge\phi^+) \,\equiv\, (\Pr j(\psi\wedge\phi))^+
\]
(this is the reason for the $\psi\wedge\phi$ clause)
and finally $\M^+\models \Pr j(\psi\wedge\phi)[s]$.
\end{proof}

In Lemma \ref{firstutilbagvalidity}, we introduced Assigned Validity as a single schema
for inclusion in $T_i$ for any $i$. In the following lemma, we need to break the
stratified version of Assigned Validity into different $\omega$-indexed families
because the stratified version of Assigned Validity intended for inclusion in $T_i$
(for any particular $i$) needs to be $i$-stratified.

\begin{lemma}
\label{secondutilbagvalidity}
(Compare Lemma \ref{firstutilbagvalidity})
For any $i,j\in\omega$, each of the following families is straticlosed-r.e.-generic.
\begin{enumerate}
\item $[S]_i$ where $S$ is: ($i$-Assigned Strativalidity)
the schema $\phi^s$ ($\phi$ valid and $i$-stratified, $s$ an assignment).
\item $[\mbox{$i$-Assigned Strativalidity}]_i \cup [\mbox{$i$-Strativalidity}]_i$.
\item $[\mbox{$i$-Assigned Strativalidity}]_i \cup
[\mbox{$i$-Validity}]_j$ (if $j\not=i$).
\end{enumerate}
\end{lemma}

\begin{proof}
For unistratifiedness, use Corollary \ref{hpreserversvalidity}.
Recursive enumerability follows from the fact that $\prec$ is r.e.
In each case below, let $\U=(U_k)_{k\in\omega}$ be a straticlosed r.e.~family
extending the family in question. For brevity, let $\M=\M_{\str\U}$.

\item
(1)
Let $I(i)$ be any computable stratifier-set above $i$, let $\phi$ be any
valid $i$-stratified formula, and let $s$ be any assignment.
Since $\phi$ is valid, $\M^{I(i)}\models \phi[s]$.
By Lemma \ref{raisingMtotheIpreservesintent}, $\M^{I(i)}\models\phi^s$,
as desired.

\item
(2)
Let $I(i)$ be any computable stratifier-set above $i$.
By (1), $\M^{I(i)}\models \mbox{$i$-Assigned Strativalidity}$.
We must show $\M^{I(i)}\models \ucl{\Prr\alpha i\phi}$, where $\phi$
is any valid $\LPA(\indset)$-formula and $\alpha<\epom$ is any ordinal
such that $\Prr\alpha i\phi$ is $i$-stratified.
Let $s$ be any assignment. Since $U_i$ contains
$i$-Assigned Strativalidity, in particular $U_i$ contains $\phi^s$.
Since $\Prr\alpha i\phi$ is $i$-stratified, $\alpha$ exceeds all the
superscripts in $\phi$ (hence in $\phi^s$), so $U_i\cap\alpha\models \phi^s$.
By definition of $\M$, this means $\M\models\Prr\alpha i\phi[s]$.
By Definition \ref{moduloidefn}, $\M$ and $\M^{I(i)}$ agree on $\Prr\alpha i$,
so $\M^{I(i)}\models\Prr\alpha i\phi[s]$, as desired.

\item
(3)
By (1), $\M^{I(i)}\models\mbox{$i$-Assigned Strativalidity}$
for every computable stratifier-set $I(i)$ above $i$.
Let $J(j)$ be a computable stratifier-set above $j$, we must
show $\M^{J(j)}\models \mbox{$i$-Validity}$.
Let $\phi$ be a valid $\LPA(\omega)$-formula, $s$
an assignment.

\item
\case1
$i\not\in\indices(J(j))$.
Then $\M^{J(j)}$ and $\M$ agree on $\Pr i$.
Let $\bullet^+$ be an $i$-stratifier.
Since $\phi$ is valid, so is $\phi^+$
(by Lemma \ref{stratifiersrespectvalidity}),
so $(\phi^+)^s\in U_i$ (since $[\mbox{$i$-Assigned Strativalidity}]_i$ is
part of line 3).
Clearly $((\phi^+)^s)^-\equiv\phi^s$, so $\phi^s\in U^-_i$,
thus $\M\models \Pr i\phi[s]$, and so does $\M^{J(j)}$.

\item
\case2
$i\in\indices(J(j))$.
Thus $j\prec i$ and we can write $J(j)=J_0\cup\{\bullet^+\}$
for some computable $i$-stratifier $\bullet^+$.
By Lemma \ref{decomposingMtotheI}, $\M^{J(j)}$ and
$\M^+$ agree on $\Pr i$.
Let $\alpha\in\epom$ be such that
$(\Pr i\phi)^+\equiv \Prr\alpha i\phi^+$.
As in Case 1, $(\phi^+)^s$ is an instance of
$i$-Assigned Strativalidity,
so $(\phi^+)^s\in U_i$
(since $[\mbox{$i$-Assigned Strativalidity}]_i$ is
part of line 3).
In fact by choice of $\alpha$, $(\phi^+)^s\in U_i\cap \alpha$,
so $\M\models \Prr\alpha i\phi^+[s]$,
that is, $\M\models (\Pr i\phi)^+[s]$.
By Lemma \ref{structurecollapsingmagic},
$\M^+\models\Pr i\phi[s]$.
Since $\M^{J(j)}$ and $\M^+$ agree on $\Pr i$,
$\M^{J(j)}\models\Pr i\phi[s]$.
\end{proof}

In Lemma \ref{secondutilbagdeduction} above, we had to modify what $T_i$
says about $j$-Deduction for
$i\prec j$. No such modification is needed in the following lemma. This is
interesting because in modal logic, positive introspection is generally
considered much more controversial and demanding than basic deduction.

\begin{lemma}
\label{secondutilbagintrospection}
(Compare Lemma \ref{firstutilbagintrospection})
For any $i,j\in\omega$, each of the following families is straticlosed-r.e.-generic.
\begin{enumerate}
\item $[\mbox{$i$-Assigned Strativalidity}]_i \cup [\mbox{$i$-Strativalidity}]_i \cup [\mbox{$i$-Stratideduction}]_i
\cup [\mbox{$i$-Introspection}]_j$ ($j\not=i$).
\item $[\mbox{$i$-Assigned Strativalidity}]_i \cup [\mbox{$i$-Strativalidity}]_i \cup [\mbox{$i$-Stratideduction}]_i
\cup [S]_i$ where $S$ is:
\[\mbox{($i$-Stratrospection) $\ucl{\Prr\alpha i\phi\rightarrow \Prr\beta
i\Prr\alpha i\phi}$
whenever this is $i$-stratified.}\]
\end{enumerate}
\end{lemma}

\begin{proof}
For unistratifiedness, use Corollary \ref{hpreserversvalidity}.
Recursive enumerability follows from the fact that $\prec$ is r.e.
In each case below, let $\U=(U_k)_{k\in\omega}$ be a straticlosed r.e.~family
extending the family in question. For brevity, let $\M=\M_{\str\U}$.

\item
(1)
By Lemma \ref{secondutilbagdeduction} (part 1)
and Lemma \ref{secondutilbagvalidity} (part 2),
$\M^{I(i)}\models(\mbox{$i$-Assigned Strativalidity})
\cup (\mbox{$i$-Strativalidity})
\cup (\mbox{$i$-Stratideduction})$
for every computable stratifier-set
$I(i)$ above $i$.
Let $J(j)$ be a computable stratifier-set above $j$, we must show
$\M^{J(j)}\models \mbox{$i$-Introspection}$.
In other words, we must show $\M^{J(j)}\models \ucl{\Pr i\phi\rightarrow \Pr i\Pr i\phi}$
for any $\LPA(\omega)$-formula $\phi$.
Let $s$ be any assignment and assume $\M^{J(j)}\models \Pr i\phi[s]$.

\item
\case1
$i\not\in\indices(J(j))$.
Then $\M^{J(j)}$ and $\M$ agree on $\Pr i$.
Thus $\M\models \Pr i\phi[s]$.
Let $\bullet^+$ be the $i$-veristratifier.
By Theorem \ref{stratificationtheorem},
$\M\models (\Pr i\phi)^+[s]$.
Let $\alpha$ be such that $(\Pr i\phi)^+\equiv \Prr\alpha i\phi^+$,
so $\M\models\Prr\alpha i\phi^+[s]$.
By definition, this means $U_i\cap\alpha\models (\phi^+)^s$.
Let $\beta$ be such that $(\Pr i\Pr i\phi)^+\equiv \Prr\beta i \Prr\alpha i\phi^+$,
so $\beta>\alpha$.
By Part 1 of Theorem \ref{proofstratification},
$U_i\cap\beta\models\Prr\alpha i (\phi^+)^s$.
Thus $\M\models \Prr\beta i\Prr\alpha i \phi^+[s]$.
By Theorem \ref{stratificationtheorem},
$\M\models (\Prr\beta i\Prr\alpha i\phi^+)^-[s]$, that is,
$\M\models \Pr i\Pr i\phi[s]$. Since $\M$ and $\M^{J(j)}$ agree on $\Pr i$,
$\M^{J(j)}\models \Pr i\Pr i\phi[s]$, as desired.

\item
\case2
$i\in\indices(J(j))$.
Thus $j\prec i$ and we can write $J(j)=J_0\cup\{\bullet^+\}$
for some computable $i$-stratifier $\bullet^+$.
By Lemma \ref{decomposingMtotheI}, $\M^{J(j)}$ and
$\M^+$ agree on $\Pr i$.
Thus $\M^+\models \Pr i\phi[s]$.
By Lemma \ref{structurecollapsingmagic},
$\M\models (\Pr i\phi)^+[s]$.
Let $\alpha$ be such that $(\Pr i\phi)^+\equiv \Prr\alpha i\phi^+$,
so $\M\models\Prr\alpha i\phi^+[s]$.
By definition of $\M$, this means
$U_i\cap\alpha\models(\phi^+)^s$.
Let $\beta$ be such that $(\Pr i \Pr i\phi)^+\equiv \Prr\beta i\Prr\alpha i\phi^+$,
so $\beta>\alpha$.
By Part 1 of Theorem \ref{proofstratification},
$U_i\cap\beta\models\Prr\alpha i(\phi^+)^s$.
Thus $\M\models\Prr\beta i\Prr\alpha i\phi^+[s]$.
In other words, $\M\models (\Pr i\Pr i\phi)^+[s]$.
By Lemma \ref{structurecollapsingmagic},
$\M^+\models \Pr i\Pr i\phi[s]$.
Since $\M^+$ and $\M^{J(j)}$ agree on $\Pr i$,
$\M^{J(j)}\models \Pr i\Pr i\phi[s]$, as desired.

\item
(2)
Let $I(i)$ be any computable stratifier-set above $i$,
we must show $\M^{I(i)}\models
\ucl{\Prr\alpha i\phi\rightarrow\Prr\beta i\Prr\alpha i\phi}$
assuming this is $i$-stratified (so $\beta>\alpha$).
Let $s$ be any assignment and assume $\M^{I(i)}\models\Prr\alpha i\phi[s]$.
By Definition \ref{moduloidefn}, $\M^{I(i)}$ and $\M$ agree on $\Prr\alpha i$,
so $\M\models\Prr\alpha i\phi[s]$.
By definition of $\M$, this means $U_i\cap\alpha\models \phi^s$.
By Part 1 of Theorem \ref{proofstratification},
$U_i\cap\beta\models \Prr\alpha i\phi^s$.
Thus, $\M\models \Prr\beta i\Prr\alpha i\phi[s]$,
and thus so does $\M^{I(i)}$ since it agrees with $\M$ on $\Prr\beta i$.
\end{proof}

For the next lemma, note that the proof shows more than is necessary,
namely that the structures in question satisfy all the axioms of Peano arithmetic
for $\LPA(\indset)$, not just the $i$-stratified ones. But of course, the full
set of Peano axioms for $\LPA(\indset)$ is not $i$-stratified.

\begin{lemma}
\label{secondutilbagarithmetic}
(Compare Lemma \ref{firstutilbagarithmetic})
For any $i\in\omega$, $[S]_i$ is straticlosed-r.e.-generic,
where $S$ is the set of those axioms of Peano arithmetic for
$\LPA(\indset)$ that are $i$-stratified.
\end{lemma}

\begin{proof}
Unistratifiedness and recursive enumerability are clear.
Let $\U=(U_k)_{k\in\omega}$ be a straticlosed r.e.~family
extending $[S]_i$.
By Lemma \ref{raisingMtotheIpreservesintent}, $\M^{I(i)}_{\str\U}$ interprets formulas
by substitution. By Lemma \ref{extendedpalemma}, $\M^{I(i)}_{\str\U}$ satisfies the axioms
of Peano Arithmetic for $\LPA(\indset)$, as desired.
\end{proof}

\begin{lemma}
\label{secondutilbagsmt}
(Compare Lemma \ref{firstutilbagsmt})
For any $i,j\in\omega$, each of the following families is straticlosed-r.e.-generic.
\begin{enumerate}
\item $[\mbox{$j$-SMT}]_i$ ($j\not=i$).
\item $[S]_i$, where $S$ is: ($i$-Strati-SMT)
$\ucl{\exists e\forall x(\Prr\alpha i\phi\leftrightarrow x\in W_e)}$
when this is $i$-stratified, $e\not\in\FV(\phi)$.
\end{enumerate}
\end{lemma}

\begin{proof}
Unistratifiedness and recursive enumerability are clear.
In each case below, let $\U=(U_k)_{k\in\omega}$ be a straticlosed r.e.~family
extending the family in question. For brevity, let $\M=\M_{\str\U}$.

\item
(1)
Let $I(i)$ be any computable stratifier-set above $i$.
We must show $\M^{I(i)}\models\ucl{\exists e\forall x(\Pr j\phi \leftrightarrow
x\in W_e)}$ for every $\LPA(\omega)$-formula $\phi$ with $e\not\in\FV(\phi)$.
Let $s$ be an assignment and let $\{x_1,\ldots,x_k\}=\FV(\phi)\backslash \{x\}$.

\item
\case1
$j\not\in \indices(I(i))$.
Then $\M^{I(i)}$ and $\M$ agree on $\Pr j$.
Since $U^-_j$ is r.e., by the $S$-$m$-$n$ theorem there is some $n$ such that
$W_n=\{m\,:\,U^-_j\models \phi(x|\overline m)(x_1|\overline{s(x_1)})\cdots
(x_k|\overline{s(x_k)})\}$.
Since $e\not\in\FV(\phi)$ and $\M$ has standard first-order part, it follows
that $\M\models \forall x(\Pr j \phi\leftrightarrow x\in W_e)[s(e|n)]$.
By first-order semantics,
$\M\models \exists e\forall x(\Pr j\phi\leftrightarrow x\in W_e)[s]$.
Since $\M$ and $\M^{I(i)}$ agree on $\Pr j$,
$\M^{I(i)}\models \exists e\forall x(\Pr j\phi\leftrightarrow x\in W_e)[s]$,
as desired.

\item
\case2
$j\in\indices(I(i))$.
Thus $i\prec j$ and we can write $I(i)=I_0\cup\{\bullet^+\}$
for some computable $j$-stratifier $\bullet^+$.
By Lemma \ref{decomposingMtotheI}, $\M^{I(i)}$ and
$\M^+$ agree on $\Pr j$.
Let $\alpha$ be such that $(\Pr j\phi)^+\equiv \Prr\alpha j\phi^+$.
Since $U_j\cap\alpha$ is r.e., by the $S$-$m$-$n$ theorem there is some $n$ such that
$W_n=\{m\,:\,U_j\cap\alpha\models\phi^+(x|\overline m)(x_1|\overline{s(x_1)})\cdots
(x_k|\overline{s(x_k)})\}$.
Since $e\not\in\FV(\phi)$ (thus $e\not\in\FV(\phi^+)$), and since $\M$ has standard
first-order part, it follows that
$\M\models \forall x(\Prr\alpha j\phi^+\leftrightarrow x\in W_e)[s(e|n)]$.
By first-order semantics,
$\M\models\exists e\forall x(\Prr\alpha j\phi^+\leftrightarrow x\in W_e)[s]$.
In other words,
$\M\models(\exists e\forall x(\Pr j\phi\leftrightarrow x\in W_e))^+[s]$.
By Lemma \ref{structurecollapsingmagic},
$\M^+\models \exists e\forall x(\Pr j\phi\leftrightarrow x\in W_e)[s]$.
Since $\M^+$ and $\M^{I(i)}$ agree on $\Pr j$,
$\M^{I(i)}\models \exists e\forall x(\Pr j\phi\leftrightarrow x\in W_e)[s]$,
as desired.

\item
(2)
Let $I(i)$ be any computable stratifier-set above $i$,
we must show $\M^{I(i)}\models\ucl{\exists e\forall x(\Prr\alpha i \phi\leftrightarrow
x\in W_e)}$ for every $\LPA(\indset)$-formula $\phi$ such that
this is $i$-stratified and $e\not\in\FV(\phi)$.
Let $s$ be any assignment
and let $\{x_1,\ldots,x_k\}=\FV(\phi)\backslash \{x\}$.
Since $U_i\cap\alpha$ is r.e., by the $S$-$m$-$n$ theorem there is some $n$
such that
$W_n=\{m\,:\,U_i\cap\alpha\models \phi(x|\overline m)(x_1|\overline{s(x_1)})\cdots
(x_k|\overline{s(x_k)})\}$.
Since $e\not\in\FV(\phi)$, and since $\M$ has standard first-order part,
it follows that
$\M\models \exists e\forall x(\Prr\alpha i\phi \leftrightarrow x\in W_e)[s]$.
By Definition \ref{moduloidefn}, $\M$ and $\M^{I(i)}$ agree on $\Prr\alpha i$,
so $\M^{I(i)}\models \exists e\forall x(\Prr\alpha i\phi \leftrightarrow x\in W_e)[s]$,
as desired.
\end{proof}

If $\T=(T_k)_{k\in\omega}$ is straticlosed-r.e.-generic, we cannot simply take
an axiom $\phi$ from $T_j$ and insert $\Pr j\phi$ into $T_i$ without violating
straticlosed-r.e.-genericness, because such a $\phi$ is not necessarily
$i$-stratified. Thus, the following lemma has a somewhat more complicated
structure than Lemma \ref{firstutilbagclosure}.

\begin{lemma}
\label{secondutilbagclosure}
(Compare Lemma \ref{firstutilbagclosure})
Let $i,j\in\omega$ and suppose $\T=(T_k)_{k\in\omega}$ is straticlosed-r.e.-generic.
Then each of the following families is straticlosed-r.e.-generic.
\begin{enumerate}
\item $\T\cup [S]_i$ where
$S$ is the schema $\Prr\alpha i\phi$ ($\phi\in T_i$ such that this is $i$-stratified).
\item $\T\cup [S]_i$ where
$S$ is the schema $\Pr j\phi^-$ ($\phi\in T_j$, $j\prec i$).
\end{enumerate}
\end{lemma}

\begin{proof}
Unistratifiedness and recursive enumerability are clear.
In each case below, let $\U=(U_k)_{k\in\omega}$ be a straticlosed r.e.~family
extending the family in question. For brevity, let $\M=\M_{\str\U}$.

\item
(1)
Since $\T$ is straticlosed-r.e.-generic and $\U\supseteq \T$ is straticlosed and r.e.,
immediately $\M^{J(j)}\models\T$ (by Definition \ref{bootstrapclosedrestrat})
for all $j\in\omega$ and any computable stratifier-set $J(j)$ above $j$.
Let $I(i)$ be any computable stratifier-set above $i$.
Suppose $\phi\in T_i$ and $\alpha\in\epom$ are such that $\Prr\alpha i\phi$
is $i$-stratified, and let $s$ be any assignment.
Since $U_i\supseteq T_i$, $\phi\in U_i$, in fact since $\Prr\alpha i\phi$ is $i$-stratified,
it follows that $\phi\in U_i\cap\alpha$. Since $\phi$ is a sentence,
$\phi\equiv \phi^s$, and so
$U_i\cap\alpha\models\phi^s$, and so $\M\models \Prr\alpha i\phi[s]$.
By Definition \ref{moduloidefn}, $\M^{I(i)}$ agrees with $\M$ on $\Prr \alpha i$,
so $\M^{I(i)}\models \Prr\alpha i\phi[s]$, as desired.

\item
(2)
Since $\T$ is straticlosed-r.e.-generic and $\U\supseteq \T$ is straticlosed and r.e.,
immediately $\M^{K(k)}\models\T$ (by Definition \ref{bootstrapclosedrestrat})
for all $k\in\omega$ and any computable stratifier-set $K(k)$ above $k$.
Let $I(i)$ be any computable stratifier-set above $i$.
Suppose $\phi\in T_j$ where $j\prec i$. Let $s$ be any assignment.
Since $U_j\supseteq T_j$, $\phi\in U_j$.
By Lemma \ref{verystratifiableaxioms}, there is some very $j$-stratified
$\psi\in U_j$ such that $\psi^-\equiv\phi^-$.
Clearly since $\phi$ is a sentence, so is $\psi$.
By compactness, there is some positive integer multiple $\alpha$ of $\epsilon_0$
such that $U_j\cap\alpha\models \psi$. Since $\psi$ is a sentence,
$\psi\equiv \psi^s$ and thus
$U_j\cap\alpha\models \psi^s$.
Thus, $\M\models\Prr\alpha j \psi[s]$.
By Theorem \ref{stratificationtheorem},
$\M\models\Pr j\psi^-[s]$, so by choice of $\psi$, $\M\models\Pr j\phi^-[s]$.
Since $I(i)$ is above $i$ and $j\not\succeq i$, $\M$ and $\M^{I(i)}$
agree on $\Pr j$, so $\M^{I(i)}\models \Pr j\phi^-[s]$, as desired.
\end{proof}

\subsection{Stratifiable-r.e.-generic Building Blocks}

We have established some straticlosed-r.e.-generic building blocks, but
the goal of this paper is to better understand the structure of non-stratified
theories---stratification is only a means to an end.
Therefore, we introduce a corresponding non-stratified building-block notion.

\begin{definition}
If $\T^0=(T^0_i)_{i\in\omega}$ where each $T^0_i$ is an $\LPA(\omega)$-theory,
we say $\T^0$ is \emph{$\prec$-stratifiable-r.e.-generic} (or
\emph{stratifiable-r.e.-generic} if $\prec$ is clear from context)
if there is
some $\prec$-straticlosed-r.e.-generic family $\T=(T_i)_{i\in\omega}$ of
$\LPA(\indset)$-theories such that
each $T^-_i=T^0_i$.
\end{definition}

\begin{lemma}
\label{trivialsourceofstratifiableregenericfamilies}
If $\T=(T_i)_{i\in\omega}$ is any straticlosed-r.e.-generic family of
$\LPA(\indset)$-theories,
then $\T^-=(T^-_i)_{i\in\omega}$ is a stratifiable-r.e.-generic family of
$\LPA(\omega)$-theories.
\end{lemma}

\begin{proof}
    Straightforward.
\end{proof}

\begin{corollary}
\label{stratifiableregenericsummary}
(Compare Corollary \ref{closedregenericsummary})
For all $i,j\in\omega$, each of the following families of $\LPA(\omega)$-theories
is stratifiable-r.e.-generic.
\begin{enumerate}
    \item $[\mbox{$j$-Deduction}]_i$ (if $j\preceq i$).
    \item $[\mbox{Modified $j$-Deduction}]_i$ (if $i\prec j$).
    \item $[\mbox{Assigned Validity}]_i$.
    \item $[\mbox{Assigned Validity}]_i \cup [\mbox{$i$-Validity}]_j$.
    \item $[\mbox{Assigned Validity}]_i
        \cup [\mbox{$i$-Validity}]_i
        \cup [\mbox{$i$-Deduction}]_i
        \cup [\mbox{$i$-Introspection}]_j$.
    \item $[S]_i$ where $S$ is the axioms of Peano Arithmetic for $\LPA(\omega)$.
    \item $[\mbox{$j$-SMT}]_i$.
    \item (If $j\preceq i$)
    $\T\cup [S]_i$, for any stratifiable-r.e.-generic
    $\T=(T_k)_{k\in\omega}$,
    where $S$ is the schema: $\Pr j\phi$ ($\phi\in T_j$).
\end{enumerate}
\end{corollary}

\begin{proof}
By combining Lemma \ref{trivialsourceofstratifiableregenericfamilies} with
Lemmas \ref{secondutilbagdeduction}--\ref{secondutilbagclosure}.
For parts involving validity, Lemma \ref{stratifiersrespectvalidity} can be
used to provide valid stratified counterparts of valid non-stratified formulas.
\end{proof}

Comparing the stratifiable-r.e.-generic families we exhibited
(Corollary \ref{stratifiableregenericsummary})
with the closed-r.e.-generic families we exhibited
(Corollary \ref{closedregenericsummary}),
we see that the stratifiable-r.e.-generic families are weaker in
exactly two ways:
\begin{enumerate}
    \item They do not
    allow $T_i$ to state $j$-Deduction for $T_j$ when
    $i\prec j$, instead allowing what we called Modified $j$-Deduction.
    \item Their closure property is more restricted:
    if $\T^1=(T^1_k)_{k\in\omega}$ is closed-r.e.-generic
    and $\T^2=(T^2_k)_{k\in\omega}$ is stratifible-r.e.-generic,
    and if $S_1$ is the schema $\Pr j \phi$ ($\phi\in T^1_j$),
    and if $S_2$ is the schema $\Pr j \phi$ ($\phi\in T^2_j$),
    then Corollary \ref{closedregenericsummary} says
    $\T^1\cup [S_1]_i$ is closed-r.e.-generic with no restrictions on $j$,
    whereas Corollary \ref{stratifiableregenericsummary} only says
    that $\T^2\cup [S_2]_i$ is stratifiable-r.e.-generic if $j\prec i$.
\end{enumerate}
We leave it an open question to what
extent Corollary \ref{stratifiableregenericsummary} could be further
strengthened. Our primary motivation in choosing building blocks was
to facilitate
creation of background provability theories at least strong enough to make our
own consistency result (Theorem \ref{generalizedtwoonethree} below) generalize
Carlson's consistency result \cite{carlson2000}.
If that were our lone motivation, we could restrict
Corollary \ref{stratifiableregenericsummary} to only those families
where $i=j$,
but a secondary motivation was to provide inter-theory versions
of those restricted building blocks.

\section[Second Consistency Result: Prioritizing Self-Truth]{Second Consistency Result:\\Prioritizing Self-Truth}

In this section, we continue to fix an r.e.\ well-founded partial-order $\prec$
of $\omega$.
The following theorem will satisfy the second promise from the introduction:
it will exhibit true theories $(T_i)_{i\in\omega}$ such that $T_i$ expresses 
a G\"odel number of $T_j$ ($j\prec i$) and the truth of $T_j$ ($j\preceq 
i$).  These theories can further be taken so that $T_i$ expresses the fact 
that $T_j$ has some G\"odel number (all $i,j$), by Lemma \ref{secondutilbagsmt}.

\begin{theorem}
\label{generalizedtwoonethree}
Let $\T^0=(T^0_i)_{i\in\omega}$ be any
stratifiable-r.e.-generic
family of $\LPA(\omega)$-theories.
For every $i\in\omega$ and $n\in\N$, let $T_i(n)$ be the smallest
$\Pr i$-closed $\LPA(\omega)$-theory containing the following axioms.
\begin{enumerate}
\item The axioms contained in $T^0_i$.
\item Assigned Validity, $i$-Validity and $i$-Deduction.
\item $\ucl{\Pr j\phi\rightarrow\phi}$ whenever $j\preceq i$.
\item $\forall x(\Pr j\phi\leftrightarrow \langle\overline{\ulcorner\phi\urcorner},\overline j,x\rangle\in W_{\overline n})$
whenever $j\prec i$, $\FV(\phi)\subseteq\{x\}$.
\end{enumerate}
Let each $\T(n)=(T_i(n))_{i\in\omega}$.
There is some $n\in\N$ such that $\T(n)$ is true.
\end{theorem}

\begin{proof}
By the $S$-$m$-$n$ Theorem, there is a total computable $f:\N\to\N$ such that $\forall n\in\N$,
\[
W_{f(n)}=\{\langle\ulcorner\phi\urcorner,j,m\rangle\in\N\,:\,\mbox{$\phi$ is an $\LPA(\omega)$-formula,
$\FV(\phi)\subseteq\{x\}$, and $T_j(n)\models\phi(x|\overline m)$}\}.
\]
By the Recursion Theorem, there is an $n\in\N$ such that $W_n=W_{f(n)}$.
We will show $\T(n)$ is true.
For the rest of the proof, we write $\T$ for $\T(n)$, $T_i$ for $T_i(n)$.

The structure of the proof is as follows.
\begin{itemize}
\item
(``Definition of $\U$'' below) First, we will define a certain
carefully-chosen family $\U=(U_i)_{i\in\omega}$ of $\LPA(\indset)$-theories
(with each $U^-_i=T_i$)
and the $\LPA(\indset)$-structure $\M=\M_{\str{\U}}$.
\item
(``Preliminary Result'' below)
Next, we will show that $\forall i\in\omega$, $\M\models U_i\cup T_i$.
In order to deal with the difficulty mentioned at
the beginning of Section \ref{stratifiersection},
we will prove more than necessary, to obtain a strong $\prec$-induction hypothesis.
Namely, we will prove, by $\prec$-induction, that $\forall i\in\omega$,
for every computable stratifier-set $I(i)$ above $i$,
$\M^{I(i)}\models U_i\cup T_i$.
    \begin{itemize}
    \item
    (Claim 1 below)
    In order to prove $\M^{I(i)}\models U_i$, we will
    use induction on $\alpha$ to show that $\M^{I(i)}\models U_i\cap\alpha$
    for all $\alpha\in\epom$.
    \item
    (Case 3 below)
    Part of proving $\M^{I(i)}\models U_i\cap\alpha$ will be proving
    $\M^{I(i)}\models\ucl{\Prr{\alpha_0} i\phi\rightarrow\phi}$
    whenever this is $i$-stratified, $\alpha_0<\alpha$.
    This is where we will use the $\alpha$-induction hypothesis.
    \item
    (Case 4 below)
    Part of proving $\M^{I(i)}\models U_i\cap\alpha$ will be proving
    $\M^{I(i)}\models\ucl{\Pr j\phi\rightarrow\phi^+}$
    whenever $j\prec i$, $\phi$ is an $\LPA(\omega)$-formula,
    and $\bullet^+$ is an $i$-stratifier. This is where we will take advantage
    of our strong $\prec$-induction hypothesis.
    \end{itemize}
\item
(Claims 2--3 below)
Once we've established $\M^{I(i)}\models U_i$, we will essentially
be able to conclude
$\M^{I(i)}\models T_i$ using the Stratification Theorem
(Theorem \ref{stratificationtheorem}).
\item
At the very end of the proof,
having established that $\forall i\in\omega$, $\M\models U_i\cup T_i$,
we will use that to prove that $\M_\T\models\T$, i.e., that $\T$ is true.
\end{itemize}

\item
\textbf{Definition of $\U$.}
Since $\T^0$ is stratifiable-r.e.-generic, there is 
a straticlosed-r.e-generic
family $\V=(V_i)_{i\in\omega}$ of $\LPA(\indset)$-theories such that each $V^-_i=T^0_i$.
For every $i\in\N$, let $U_i$ be the smallest
$i$-stratified $\LPA(\indset)$-theory such that the following hold.
\begin{enumerate}
\item $U_i$ contains $V_i$.
\item $U_i$ contains $i$-Assigned Strativalidity, $i$-Strativalidity, $i$-Stratideduction and $i$-Collapse.
\item $U_i$ contains $\ucl{\Prr\alpha i\phi\rightarrow\phi}$ whenever $\Prr\alpha i\phi$ is $i$-stratified.
\item $U_i$ contains $\ucl{\Pr j\phi\rightarrow\phi^+}$ for every $\LPA(\omega)$-formula $\phi$, $j\prec i$,
and $i$-stratifier $\bullet^+$.
\item $U_i$ contains $\forall x(\Pr j\phi\leftrightarrow\langle\overline{\ulcorner\phi\urcorner},\overline j,x\rangle\in W_{\overline n})$
whenever $j\prec i$, $\FV(\phi)\subseteq\{x\}$ and $\phi$ is an $\LPA(\omega)$-formula.
\item Whenever $\phi\in U_i$ and $\Prr\alpha i\phi$ is $i$-stratified, $\Prr\alpha i\phi\in U_i$.
\end{enumerate}
Let $\U=(U_i)_{i\in\omega}$.  Observe that $\U$ is straticlosed and r.e.~(to see $U_i$ is
$i$-unistratified, use
Lemma \ref{stratifiermagic}; to see $\U$ is r.e., use
Theorem \ref{blackbox} part 1); $\U\supseteq\V$; and for each $i\in\omega$,
$U_i^-=T_i$.

Let $\M=\M_{\str{\U}}$.
Recall that
$\str{\U}$ is the $\LPA(\indset)$-family $(S_i)_{i\in\indset}$
where $\forall i\in\omega$ and $\alpha\in\epom$, $S_i=U^-_i=T_i$ and
$S_{(\alpha,i)}=U_i\cap \alpha$.
For the reader's convenience, here is how (by definition)
$\M$ interprets $\Pr i$ and $\Prr \alpha i$ for all $i\in\omega$, $\alpha\in\epom$:
\begin{align*}
    \M\models \Pr i \phi[s] &\mbox{ iff } T_i\models \phi^s,\\
    \M\models \Prr \alpha i \phi[s] &\mbox{ iff } U_i\cap \alpha\models\phi^s.
\end{align*}

\item
\textbf{Preliminary Result.}
We would like to prove the following preliminary result:
$\forall i\in\omega$, $\M\models U_i\cup T_i$.
For the sake of a stronger induction hypothesis, we will prove that
$\forall i\in\omega$, for every computable stratifier-set $I(i)$ above $i$,
$\M^{I(i)}\models U_i\cup T_i$.
This is more than enough because $\M^{I(i)}=\M$ when
$I(i)=\emptyset$.

Fix $i\in\omega$. By $\prec$-induction, we have the following:
\begin{quote}
($*$) For every $j\prec i$, for every computable stratifier-set $J(j)$ above
$j$, $\M^{J(j)}\models U_j\cup T_j$.
\end{quote}

Let $I(i)$ be any computable stratifier-set above $i$.
We must show $\M^{I(i)}\models U_i\cup T_i$.

\item
\claim1
$\forall \alpha\in\epom$, $\M^{I(i)}\models U_i\cap\alpha$.

\item
By induction on $\alpha$.
Let $\sigma\in U_i\cap\alpha$.

\item
\case1
$\sigma\in V_i$.
Then $\M^{I(i)}\models\sigma$ because
$\V$ is straticlosed-r.e.-generic and $\U\supseteq\V$
is straticlosed and r.e.

\item
\case2
$\sigma$ is an instance of $i$-Assigned Strativalidity, $i$-Strativalidity,
or $i$-Stratideduction.  Then $\M^{I(i)}\models\sigma$
by Lemma \ref{secondutilbagdeduction} or Lemma \ref{secondutilbagvalidity}.

\item
\case3
$\sigma$ is $\ucl{\Prr{\alpha_0} i\phi\rightarrow\phi}$ for some
$i$-stratified $\LPA(\indset)$-formula $\phi$ such that $\Prr{\alpha_0}i\phi$ is $i$-stratified.
Since $\sigma\in U_i\cap\alpha$, this forces $\alpha_0<\alpha$.
Let $s$ be an assignment and
assume $\M^{I(i)}\models\Prr{\alpha_0}i\phi[s]$, then:
\begin{align*}
\M^{I(i)} &\models \Prr{\alpha_0} i\phi[s]
  &\mbox{(Assumption)}\\
\M &\models \Prr{\alpha_0} i\phi[s]
  &\mbox{($\M$ and $\M^{I(i)}$ agree on $\Prr{\alpha_0} i$
  by Def.\ \ref{moduloidefn})}\\
U_i\cap \alpha_0 &\models \phi^s
  &\mbox{(Definition of $\M$)}\\
\M^{I(i)} &\models \phi^s
  &\mbox{(By $\alpha$-induction, $\M^{I(i)}\models U_i\cap \alpha_0$)}\\
\M^{I(i)} &\models \phi[s].
  &\mbox{(Lemma \ref{raisingMtotheIpreservesintent})}
\end{align*}

\item
\case4
$\sigma$ is $\ucl{\Pr j\phi\rightarrow\phi^+}$ for some $\LPA(\omega)$-formula $\phi$, $j\prec i$,
and $i$-stratifier $\bullet^+$.
By Lemma \ref{stratifiermagic} we may assume $\bullet^+$ is computable.
Let $J(j)$ be the computable stratifier-set $J(j)=I(i)\cup\{\bullet^+\}$, which is above $j$
since $I(i)$ is above $i$ and $j\prec i$.
Let $s$ be an assignment and assume $\M^{I(i)}\models \Pr j\phi[s]$, then:
\begin{align*}
\M^{I(i)} &\models \Pr j\phi[s]
  &\mbox{(Assumption)}\\
\M &\models \Pr j\phi[s]
  &\mbox{(Since $j\prec i$ and $I(i)$ is above $i$, $\M^{I(i)}$ and $\M$ agree
  on $\Pr j$)}\\
T_j &\models \phi^s
  &\mbox{(Definition of $\M$)}\\
\M^{J(j)} &\models \phi^s
  &\mbox{(Since $\M^{J(j)}\models T_j$ by ($*$))}\\
(\M^{I(i)})^+ &\models \phi^s
  &\mbox{(Lemma \ref{decomposingMtotheI})}\\
\M^{I(i)} &\models (\phi^s)^+
  &\mbox{(Lemma \ref{structurecollapsingmagic})}\\
\M^{I(i)} &\models (\phi^+)^s
  &\mbox{(Clearly $(\phi^s)^+\equiv(\phi^+)^s$)}\\
\M^{I(i)} &\models \phi^+[s].
  &\mbox{(Lemma \ref{raisingMtotheIpreservesintent})}
\end{align*}

\item
\case5
$\sigma$ is
$\forall x(\Pr j\phi\leftrightarrow\langle\overline{\ulcorner\phi\urcorner},\overline j,
x\rangle\in W_{\overline n})$
for some $\LPA(\omega)$-formula $\phi$ with $\FV(\phi)\subseteq\{x\}$
and $j\prec i$.
Let $s$ be any assignment, say $s(x)=m$.
The following biconditionals are equivalent:
\begin{align*}
\M^{I(i)} \models \Pr j\phi&\leftrightarrow
\langle\overline{\ulcorner\phi\urcorner}, \overline j,x\rangle\in W_{\overline n}[s]\\
\M \models \Pr j\phi&\leftrightarrow
\langle\overline{\ulcorner\phi\urcorner}, \overline j,x\rangle\in W_{\overline n}[s]
  &\mbox{($\M^{I(i)}$ and $\M$ agree on the symbols in question)}\\
\M \models \Pr j\phi[s]&\mbox{ iff }\M\models
\langle\overline{\ulcorner\phi\urcorner}, \overline j,\overline m\rangle\in W_{\overline n}
  &\mbox{(Lemma \ref{raisingMtotheIpreservesintent})}\\
\M \models \Pr j\phi[s]&\mbox{ iff }
\langle\ulcorner\phi\urcorner,j,m\rangle\in W_n
  &\mbox{($\M$ has standard first-order part)}\\
T_j \models \phi^s&\mbox{ iff }
\langle\ulcorner\phi\urcorner,j,m\rangle\in W_n
  &\mbox{(Definition of $\M$)}\\
T_j \models \phi(x|\overline m)&\mbox{ iff }
\langle\ulcorner\phi\urcorner,j,m\rangle\in W_n.
  &\mbox{(Since $\FV(\phi)\subseteq\{x\}$)}
\end{align*}
The latter is true by definition of $n$.

\item
\case6
$\sigma$ is an instance $\Prr\beta i\phi\leftrightarrow\Prr\gamma i\phi$
of $i$-Collapse (so $\beta\leq_1\gamma$ and $\Prr\beta i\phi\leftrightarrow\Prr\gamma i\phi$ is
$i$-stratified).
Let $s$ be an assignment,
since $\M^{I(i)}$ and $\M$ agree
on $\Prr\beta i$ and $\Prr\gamma i$, we need only show
$\M\models\Prr\beta i\phi\leftrightarrow\Prr\gamma i\phi[s]$.
In other words we must show $U_i\cap\beta\models\phi^s$
if and only if $U_i\cap\gamma\models\phi^s$.
This is by Theorem \ref{collapsetheorem}.

\item
\case7
$\sigma$ is $\Prr{\alpha_0}i\phi$ for some $\LPA(\indset)$-formula $\phi$
such that $\Prr{\alpha_0}i\phi$ is $i$-stratified and $\phi\in U_i$.
Since $\Prr{\alpha_0}i\phi$
is $i$-stratified, $\onset(\phi)\subseteq\alpha_0$, so $\phi\in U_i\cap\alpha_0$.
Thus $\M\models \Prr{\alpha_0}i\phi$,
so $\M^{I(i)}\models \Prr{\alpha_0}i\phi$ since
$\M^{I(i)}$ and $\M$ agree on $\Prr{\alpha_0}i$.

\item
Cases 1--7
establish $\M^{I(i)}\models U_i\cap\alpha$.
By arbitrariness of $\alpha$, Claim 1 is proved.

\item
\claim2
For any assignment $s$
and any very $i$-stratified $\LPA(\indset)$-formula $\phi$,
$\M^{I(i)}\models\phi[s]$ if and only if
$\M^{I(i)}\models\phi^-[s]$.

\item
By Theorem \ref{stratificationtheorem}, for all such $s$ and $\phi$,
$\M\models\phi[s]$ if and only if $\M\models\phi^-[s]$.
The claim now follows from Lemma \ref{technicalLemmaToShortenMainProof}
($i\not\in\indices(I(i))$ because $I(i)$ is above $i$).

\item
\claim3
$\M^{I(i)}\models T_i$.

\item
For any $\sigma\in T_i$,
there is some $\tau\in U_i$ such that $\tau^-\equiv\sigma$;
since $U_i$ is $i$-unistratified, we may take $\tau$ to be very $i$-stratified
(Lemma \ref{verystratifiableaxioms}).
By Claim 1, $\M^{I(i)}\models U_i$,
so $\M^{I(i)}\models\tau$.
By Claim 2, $\M^{I(i)}\models\sigma$.

\item
For each $i\in\omega$, letting $I(i)=\emptyset$,
Claims 1--3 show that
$\M\models U_i\cup T_i$.
It follows that $\M\models \T$.
Now, for every $i\in\omega$, $\M_\T$ interprets $\Pr i$ as follows:
\[
\M_\T\models \Pr i \phi[s] \mbox{ iff } T_i\models \phi^s.
\]
This is exactly the same way that $\M$ interprets $\Pr i$.
It follows that $\M$ and $\M_\T$ agree on $\LPA(\omega)$-formulas.
Thus, since $\M\models\T$, $\M_\T\models \T$, i.e., $\T$ is true.
\end{proof}

\section{Well-Foundation and Ill-Foundation}

The following is a variation on Kleene's $\O$.

\begin{definition}
\label{kleeneodefn}
Simultaneously define $\O\subseteq\N$
and $|\bullet|:\O\to\mathrm{Ord}$ so that
$\O\subseteq\N$
is the smallest set such that:
\begin{enumerate}
\item $0\in\O$ (it represents the ordinal $|0|=0$).
\item $\forall n\in\O$, $2^n\in \O$ (it represents the ordinal $|2^n|=|n|+1$).
\item If $\varphi_e$ (the $e$th partial recursive function) is total
and $\mathrm{range}(\varphi_e)\subseteq\O$,
then $3\cdot 5^e\in\O$
(it represents the ordinal $|3\cdot 5^e|=\sup\{|\varphi_e(0)|,|\varphi_e(1)|,\ldots\}$).
\end{enumerate}
\end{definition}

To avoid technical complications,
we have differed from the usual Kleene's $\O$ in the following way:
in the usual definition, in order for $3\cdot5^e$ to lie in $\O$,
it is also required that $|\varphi_e(0)|<|\varphi_e(1)|<\cdots$.

\begin{definition}
$\LPO$ is the language of Peano arithmetic extended by
a unary predicate $\O$.
The following notions are defined by analogy with Section \ref{prelimsect}:
\begin{enumerate}
\item For any assignment $s$ and $\LPO(I)$-formula $\phi$ with $\FV(\phi){=}\{x_1,\ldots,x_n\}$,
$\phi^s\equiv\phi(x_1|\overline{s(x_1)})\cdots(x_n|\overline{s(x_n)})$.
\item If $\T=(T_i)_{i\in I}$ is an $I$-indexed family of $\LPO(I)$-theories,
the \emph{intended structure} for $\T$ is the $\LPO(I)$-structure
$\M_\T$ with universe $\N$, interpreting symbols of $\mathrm{PA}$ as usual
and interpreting $\O$ as $\mathcal O$, and interpreting $\Pr i$ ($i\in I$)
as in Definition \ref{intendedstructdefn}.
For any $\LPO(I)$-structure $\mathscr N$, we write $\mathscr N\models \T$
if $\forall i\in I$, $\mathscr N\models T_i$.
We say $\T$ is \emph{true} if $\M_\T\models\T$.
\end{enumerate}
\end{definition}

\begin{definition}
If $I$ is an index set and 
$\T=(T_i)_{i\in I}$ is a family of $\LPO(I)$-theories,
then for any $i\in I$
such that $\M_\T\models T_i$,
we define the ordinal $\|T_i\|=\sup\{|m|+1\,:\,T_i\models\O(\overline m)\}$.
\end{definition}

The above definition makes sense:
since $\M_\T\models T_i$
and $\O^{\M_\T}=\O$,
the supremands are defined.

\begin{definition}
\label{basicaxiomsofodefn}
The \emph{basic axioms of $\O$} are the following $\LPO$-axioms.
\begin{enumerate}
\item $\O(0)$.
\item $\O(\overline n)\rightarrow \O(\overline{2^n})$, for every $n\in\N$.
\item $\forall x(\varphi_{\overline n}(x){\downarrow}\mathrel{\&}\O(\varphi_{\overline n}(x)))
\rightarrow
\O(\overline{3\cdot 5^n})$, for every $n\in\N$.
\end{enumerate}
\end{definition}

We have written the last two lines using infinite schemata
to strengthen the following result.

\begin{theorem}
\label{olanguagestructurethm}
Let $I$ be an index set,
$\prec$ a binary relation on $I$.
Suppose $\T=(T_i)_{i\in I}$
is a family of $\LPO(I)$-theories
with the following properties:
\begin{enumerate}
\item $\forall i\in I$, $T_i$ contains the axioms of Peano arithmetic.
\item $\forall i\in I$, $T_i$ contains the basic axioms of $\O$.
\item $\forall i\in I$, $\forall j\prec i$, $\exists n\in\N$
such that $T_i\models\forall x(\Pr j\O(x)\leftrightarrow x\in W_{\overline n})$.
\item $\forall i\in I$, $\forall j\prec i$, $T_i\models \forall x(\Pr j\O(x)\rightarrow\O(x))$.
\end{enumerate}
If $\M_\T\models T_i\cup T_j$ (in particular if $\T$ is true)
and $j\prec i$, then $\|T_j\|<\|T_i\|$.
\end{theorem}

\begin{proof}
Assume $\M_\T\models T_i\cup T_j$ and $j\prec i$.
By hypothesis
there is some $n\in\N$ such that $T_i\models\forall x(\Pr j\O(x)\leftrightarrow x\in W_{\overline n})$
and $T_i\models\forall x(\Pr j\O(x)\rightarrow\O(x))$.
From these, $T_i\models\forall x(x\in W_{\overline n}\rightarrow \O(x))$.

Since $\M_\T\models T_i$,
in particular $\M_\T\models \forall x(\Pr j\O(x)\leftrightarrow x\in W_{\overline n})$.
This means $W_n=\{m\in\N\,:\,T_j\models\O(\overline m)\}$.
Since $T_j$ includes the axiom $\O(0)$,
$W_n\not=\emptyset$.

Since $W_n\not=\emptyset$,
by computability theory
there is some $k\in\N$ such that
\[
\mathrm{PA} \models (\mathrm{domain}(\varphi_{\overline k})=\N)\wedge (\mathrm{range}(\varphi_{\overline k})=W_{\overline n}). 
\]
Since $T_i$ includes $\mathrm{PA}$,
$T_i$ also implies as much.
Combined with
$T_i\models\forall x(x\in W_{\overline n}\rightarrow\O(x))$,
it follows that
$T_i\models
\forall x(\varphi_{\overline{k}}(x){\downarrow} \mathrel{\&} \O(\varphi_{\overline{k}}(x)))$.
Since $T_i$ contains
the basic axiom
$\forall x(\varphi_{\overline k}(x){\downarrow}
\mathrel\& \O(\varphi_{\overline k}(x)))\rightarrow\O(\overline{3\cdot 5^k})$,
$T_i\models \O(\overline{3\cdot 5^k})$.

To finish the proof, calculate
\begin{align*}
\|T_j\|
&=
\sup\{|m|+1\,:\,T_j\models\O(\overline m)\}\\
&=
\sup\{|m|\,:\,T_j\models\O(\overline m)\}
  &\mbox{(Since $T_j$ contains $\O(\overline n)\rightarrow\O(\overline{2^n})$ for all $n\in\N$)}\\
&=
\sup\{|m|\,:\,m\in W_n\}
  &\mbox{(Since $W_n=\{m\in\N\,:\,T_j\models \O(\overline m)\}$)}\\
&=
\sup\{|\varphi_k(0)|,|\varphi_k(1)|,\ldots\}
  &\mbox{(By choice of $k$)}\\
&=
|3\cdot 5^k|
  &\mbox{(Definition \ref{kleeneodefn})}\\
&<
\sup\{|m|+1\,:\,T_i\models \O(\overline m)\}
  &\mbox{(Since $T_i\models \O(\overline{3\cdot 5^k})$)}\\
&=
\|T_i\|.
\end{align*}
\end{proof}

\begin{corollary}
\label{olanguagewellfounded}
(Well-Foundedness of True Self-Referential Theories)
Let $I$, $\T$, $\prec$ be as in Theorem \ref{olanguagestructurethm}.
If $\T$ is true then $\prec$ is well founded,
by which we mean there is no infinite descending sequence
$i_0\succ i_1\succ\cdots$.
\end{corollary}

In particular
Corollary \ref{olanguagewellfounded} says that
if $I$, $\T$, $\prec$ are as in Theorem \ref{olanguagestructurethm} and $\T$ is true
then $\prec$ is strict: there is no $i$ with $i\prec i$.
This gives a new form
(under the additional new assumption of containing/knowing basic rudiments of computable ordinals)
of the Lucas--Penrose--Reinhardt argument that a truthful theory (or machine)
cannot state (or know) its own truth and its own G\"odel number.

We could remove Peano arithmetic from Theorem \ref{olanguagestructurethm}
if we further departed from Kleene and changed line 3 of Definition 
\ref{kleeneodefn} to read:
\begin{enumerate}
\setcounter{enumi}{2}
\item If $W_e\subseteq\O$, then $3\cdot 5^e\in\O$ (and 
$|3\cdot 5^e|=\sup\{|n|\,:\,n\in W_e\}$, or $|3\cdot 5^e|=0$
if $W_e=\emptyset$)
\end{enumerate}
(and altered Definition \ref{basicaxiomsofodefn} accordingly).
The previous paragraph would still stand, in fact giving a version
of the Lucas--Penrose--Reinhardt argument in which the theory (machine)
is not required to contain (know) arithmetic.

We close the paper by showing that Corollary \ref{olanguagewellfounded}
fails without $\O$.  Let $\mathrm{WF}$ be the set
of all r.e.~well-founded partial orders on $\omega$
and let $\Tr$ be the set of all
true $\LPA$-sentences.
It is well-known that $\mathrm{WF}$ is computability theoretically
$\Pi^1_1$-complete and $\Tr$ is $\Delta^1_1$, so
$\mathrm{WF}$ cannot be defined
in $\LPA\cup\{\Tr\}$.

\begin{theorem}
\label{illfoundednesstheorem}
(Ill-Foundedness of True Self-Referential Theories)
\begin{enumerate}
\item
There exists an r.e., ill-founded partial order $\prec$ on $\omega$
such that for every closed-r.e.-generic $\T^0=(T^0_i)_{i\in\omega}$
there is an $n\in\N$ such that $\T(n)$ is true,
where $\T(n)$ is as in Theorem \ref{onethreethree}.
\item
There exists an r.e., ill-founded partial order $\prec$ on $\omega$
such that for every $\prec$-stratifiable-r.e.-generic $\T^0=(T^0_i)_{i\in\omega}$
there is an $n\in\N$ such that $\T(n)$ is true,
where $\T(n)$ is as in Theorem \ref{generalizedtwoonethree}.
\end{enumerate}
\end{theorem}

\begin{proof}
We prove (1), (2) is similar.
Assume $\neg(1)$.
For each r.e.\ partial order $\prec$ on $\omega$, let $S(\prec)$ be
the statement of Theorem \ref{onethreethree} for $\prec$,
minus the requirement
that $\prec$ be well founded. Combining $\neg(1)$ with Theorem
\ref{onethreethree}, $\prec$ is well founded if and only if
$S(\prec)$ is true. We will argue that $S(\prec)$ is expressible in
$\LPA\cup\{\Tr\}$, which is absurd because that would mean it is
possible to define $\mathrm{WF}$ in $\LPA\cup\{\Tr\}$.

$S(\prec)$ is equivalent to the following:
\begin{itemize}
    \item
    For any (G\"odel number of an) r.e.\ family $\T^0=(T^0_i)_{i\in\omega}$
    of $\LPA(\omega$)-theories,
    if $\T^0$ is closed-r.e.-generic
    (i.e., if $\M_\U\models\T^0$ for every closed r.e.\ family $\U\supseteq\T^0$
    of $\LPA(\omega)$-theories), then there is some $n\in\mathbb N$
    such that $\T(n)$ is true
    (i.e., such that $\M_{\T(n)}\models\T(n)$), where
    $\T(n)=(T_i(n))_{i\in\omega}$, where each $T_i(n)$ is the smallest
    $\Pr i$-closed theory containing the following:
    \begin{enumerate}
        \item
        The axioms in $T^0_i$.
        \item
        $\forall x (\Pr j \phi \leftrightarrow \langle \overline{\ulcorner \phi\urcorner},
        \overline j, x\rangle\in W_{\overline n})$ whenever $j\in\omega$,
        $\FV(\phi)\subseteq\{x\}$.
        \item
        $\ucl{\Pr j \phi\rightarrow\phi}$ whenever $j\prec i$.
    \end{enumerate}
\end{itemize}
This is manifestly expressible in $\LPA$ except for the clauses
$\M_\U\models \T^0$ and $\M_{\T(n)}\models \T(n)$.
We will show that $\M_\U\models \T^0$ is expressible in $\LPA\cup\{\Tr\}$;
the expressibility of $\M_{\T(n)}\models \T(n)$ is similar.

Define an operator $F_\U$ which takes an $\LPA(\omega)$-formula $\phi$ and outputs
an $\LPA$-formula $F_\U(\phi)$ as follows:
\begin{itemize}
    \item
    If $\phi$ is atomic, let $F_\U(\phi)\equiv \phi$.
    \item
    If $\phi$ is $\neg\phi_0$, $\phi_1\rightarrow\phi_2$,
    or $\forall x\phi_0$, let $F_\U(\phi)$
    be $\neg F_\U(\phi_0)$, $F_\U(\phi_1)\rightarrow F_\U(\phi_2)$,
    or $\forall x F_\U(\phi_0)$, respectively.
    \item
    Suppose $\phi$ is $\Pr i \psi$ and $\FV(\psi)=\{x_1,\ldots,x_k\}$.
    Let $f:\mathbb N^k\to\mathbb N$ be the computable function such that
    for all $m_1,\ldots,m_k\in\mathbb N$,
    $f(m_1,\ldots,m_k)
    =
    \ulcorner\psi(x_1|\overline{m_1})\cdots (x_k|\overline{m_k})\urcorner$.
    Let $F_\U(\phi)$ be: ``$U_i$ proves
    the sentence with G\"odel number $f(x_1,\ldots,x_k)$''
    (so $\FV(F_\U(\phi))=\{x_1,\ldots,x_k\}$).
\end{itemize}
It is easy to check that for every $\LPA(\omega)$-formula $\phi$ and assignment $s$,
$\M_\U\models\phi[s]$ if and only if $\mathbb N\models F_\U(\phi)[s]$.
In particular, for every $\LPA(\omega)$-sentence $\phi$, $\M_\U\models\phi$
if and only if $\mathbb N\models F_\U(\phi)$. Thus,
the clause $\M_\U\models \T^0$ can be expressed
in $\LPA\cup\{\Tr\}$ as follows:
$\forall i\forall x(x\in T^0_i\rightarrow \Tr(\ulcorner F_\U(x)\urcorner))$.
\end{proof}

The way we prove Theorem \ref{illfoundednesstheorem} by referring to the
computability theoretical complexity of $\mathrm{WF}$ is similar to a
recent argument by Kripke \cite{kripke}.

\end{document}